\documentclass{article}
\usepackage{amsmath,amsfonts,latexsym, amsrefs,amssymb}

\usepackage{graphicx}

\DeclareMathOperator{\Prob}{\mathbb{P}}   

\newcommand{\1}{\mathds{1}}
\usepackage{enumerate}

\usepackage{wrapfig}

\usepackage{color}

\numberwithin{equation}{section}

\newcommand{\rd}{{\rm d}}
\newcommand{\rdA}{{\rm d A}}

\newcommand{\pBB}[1]{\Biggl({#1}\Biggr)}

\newcommand{\ba}{{\bf{a}}}

\newcommand{\by}{{\bf{y}}}
\newcommand{\bu}{{\bf{u}}}
\newcommand{\bv}{{\bf{v}}}

\newcommand{\al}{\alpha}

\newcommand{\be}{\begin{equation}}
\newcommand{\ee}{\end{equation}}

\newcommand{\hp}[1]{with $ #1 $-high probability}

\newcommand{\e}{{\varepsilon}}

\newcommand{\T}{\mathbb T}

\newcommand{\bT}{\T}
\newcommand{\non}{\nonumber}
\newcommand{\wH}{K}

\newcommand{\ttau}{\vartheta}

\usepackage{amsmath} 
\usepackage{amssymb}
\usepackage{amsthm}
\usepackage{amsxtra}
\usepackage{amscd}
\usepackage{bbm}
\usepackage{mathrsfs}
\usepackage{bm}

\renewcommand{\b}[1]{\bm{\mathrm{#1}}} 
\newcommand{\bb}{\mathbb} 

\renewcommand{\cal}{\mathcal}

\newcommand{\wt}{\widetilde}
\newcommand{\mG}{\mathcal G}

\newcommand{\me}{\mathrm{e}} 
\newcommand{\ii}{\mathrm{i}} 

\newcommand{\col}{\mathrel{\mathop:}}
\newcommand{\st}{\,\col\,}
\newcommand{\deq}{\mathrel{\mathop:}=}

\renewcommand{\epsilon}{\varepsilon}
\renewcommand{\leq}{\leqslant}
\renewcommand{\geq}{\geqslant}

\renewcommand{\le}{\leq}
\renewcommand{\ge}{\geq}

\renewcommand{\P}{\mathbb{P}}
\newcommand{\E}{\mathbb{E}}
\newcommand{\R}{\mathbb{R}}
\newcommand{\C}{\mathbb{C}}

\newcommand{\qB}[1]{\Bigl[{#1}\Bigr]}

\newcommand{\hb}[1]{\bigl\{{#1}\bigr\}}

\newcommand{\hbb}[1]{\biggl\{{#1}\biggr\}}

\newcommand{\abs}[1]{\lvert #1 \rvert}
\newcommand{\absb}[1]{\bigl\lvert #1 \bigr\rvert}

\DeclareMathOperator{\tr}{Tr}

\DeclareMathOperator{\re}{Re}
\DeclareMathOperator{\im}{Im}

\DeclareMathOperator{\sgn}{sgn}

\DeclareMathOperator{\OO}{O}

\theoremstyle{plain} 
\newtheorem{theorem}{Theorem}[section]
\newtheorem*{theorem*}{Theorem}
\newtheorem{lemma}[theorem]{Lemma}
\newtheorem*{lemma*}{Lemma}
\newtheorem{corollary}[theorem]{Corollary}
\newtheorem*{corollary*}{Corollary}
\newtheorem{proposition}[theorem]{Proposition}
\newtheorem*{proposition*}{Proposition}

\newtheorem{definition}[theorem]{Definition}
\newtheorem*{definition*}{Definition}

\newtheorem*{example*}{Example}

\newtheorem*{remark*}{Remark}
\newtheorem*{remarks*}{Remarks}

\makeatletter
\renewcommand{\section}{\@startsection
{section}
{1}
{0mm}
{-2\baselineskip}
{1\baselineskip}
{\normalfont\large\scshape\centering}} 
\makeatother

\makeatletter
\renewcommand{\subsection}{\@startsection
{subsection}
{2}
{0mm}
{-\baselineskip}
{0\baselineskip}
{\normalfont\bf} } 
\makeatother

\usepackage[T1]{fontenc}

\usepackage{dsfont}
\usepackage{stmaryrd}

\newcommand{\nc}{\normalcolor}

\begin{document}

\title{{\sc\LARGE Local Circular Law for Random Matrices}\vspace{1cm}}

\date{}

\author{\vspace{0.5cm}\normalsize{\sc Paul Bourgade}${}^1$\thanks{Partially supported by NSF grant DMS-1208859}\quad\quad
{\sc Horng-Tzer Yau}${}^1$\thanks{Partially supported
by NSF grants DMS-0757425, 0804279}\quad\quad
{\sc Jun Yin}${}^2$\thanks{Partially supported
by NSF grant DMS-1001655}
 \\\\
\normalsize Department of Mathematics, Harvard University\\
\normalsize Cambridge MA 02138, USA \\ \normalsize  bourgade@math.harvard.edu \quad
\normalsize htyau@math.harvard.edu ${}^1$ \quad  \\ \\
\normalsize Department of Mathematics, University of Wisconsin-Madison \\
\normalsize Madison, WI 53706-1388, USA \ \normalsize jyin@math.wisc.edu ${}^2$\vspace{1cm}}

\maketitle

\begin{abstract}
The circular law asserts  that the spectral measure
of  eigenvalues of  rescaled random matrices without symmetry assumption converges to the uniform measure on the unit disk.
We prove a local version of this law at any point $z$ away from the unit circle. More precisely,  if $ | |z| - 1 | \ge \tau$
for arbitrarily small $\tau> 0$,
the circular law is valid around $z$ up to scale $N^{-1/2+ \e}$ for any $\e > 0$ under the  assumption that
the distributions of the matrix entries satisfy a uniform subexponential decay condition.
\end{abstract}

\vspace{1.5cm}

{\bf AMS Subject Classification (2010):} 15B52, 82B44

\medskip

\medskip

{\it Keywords:} local circular law, universality.

\medskip

\newpage

 \section{Introduction}
A considerable literature about random matrices focuses on Hermitian or symmetric  matrices with independent entries.
These  models  are paradigms for
local eigenvalues statistics of many random Hamiltonians, as envisioned by Wigner.
The  study of non-Hermitian random matrices
goes back to Ginibre, then in Princeton and motivated by Wigner.
Ginibre's  viewpoint on the problem was described as follows  \cite{Gin1965}:

{\it Apart  from  the
intrinsic  interest  of  the   problem,
one  may  hope
that   the   methods  and  results  will  provide  further
insight  in the   cases  of  physical  interest  or suggest
as  yet  lacking applications.}

In fact the eigenvalues statistics found by Ginibre, in the case of Gaussian complex or real  entries,
correspond to bidimensional gases, with distinct temperatures and symmetry conditions; this is therefore
a model for many interacting particle systems in dimension 2 (see e.g. \cite{For2010} chap. 15).
The spectral statistics found in \cite{Gin1965} in the complex case are the following:  given a
$N\times N$ matrix with independent  entries $\frac{1}{\sqrt {N}}z_{ij}$, the $z_{ij}$'s being
identically distributed according to the standard complex Gaussian measure $\mu_g=\frac{1}{\pi}e^{-|z|^2}\rdA(z)$
(where $\rdA$ denotes the Lebesgue measure on $\mathbb{C}$),
its eigenvalues $\mu_1,\dots,\mu_N$ have a probability density
proportional to
$$
\prod_{i<j}|\mu_i-\mu_j|^2e^{-N\sum_{k}|\mu_k|^2},
$$
with respect to the Lebesgue measure on $\C^{N}$.
This law is  a determinantal point process (because of the Vandermonde determinant) with an
explicit kernel
given by (see \cite{Gin1965,Meh2004} for a proof)
$$
K_N(z_1,z_2)=\frac{N}{\pi}e^{-\frac{N}{2}(|z_1|^2+|z_2|^2)}\sum_{\ell=0}^{N-1}\frac{(N z_1\overline{z_2})^\ell}{\ell!},
$$
with respect to the Lebesgue measure on $\C$. This integrability property allowed Ginibre to derive the  circular law for the eigenvalues,
i.e.,
the empirical spectral distribution converges to the uniform measure on the unit circle,
\begin{equation}\label{eqn:circular}
\frac{1}{\pi}\1_{|z|<1}\rdA(z).
\end{equation}
This phenomenon is the  non-Hermitian counterpart of the semicircular law for Wigner random Hermitian matrices,
and the quarter circular limit for Marchenko-Pastur random covariance matrices.

In the case of real Gaussian entries, the join distribution of the eigenvalues is
more complicated but still integrable, allowing Edelman \cite{Ede1997}
to prove the limiting circular law as well; for more precise asymptotic
properties of the real Ginibre ensemble, see \cite{ForNag2007,Sin2007,BorSin2009}.
We note also that the (right) eigenvalues of the quaternionic Ginibre ensemble
were recently shown to converge to a (non-uniform) measure on the unit ball of the quaternions field \cite{BenCha2011}.

 For non-Gaussian entries, there is no explicit formula for the eigenvalues. Furthermore, the spectral measure,
as a measure on $\C$,  cannot be   characterized by computing $\tr(M^\alpha\bar M^\beta)$. Thus the moment method,
which is the popular  way  to prove the semicircle law,
cannot be applied to solve  this problem.
Nevertheless,
Girko \cite{Gir1984}  partially proved
that
 the spectral measure
of a non-Hermitian matrix $M$ with independent
entries  converges to the circular
law (\ref{eqn:circular}). The key  insight of this work was the
introduction of the {\it Hermitization technique}.
This allows him to
translate the convergence of complex empirical measures into the convergence of logarithmic transforms for
a family of Hermitian matrices.  More precisely, if we denote the original non-Hermitian matrix by $X$ and the eigenvalues of
$X$ by $\mu_j$,  then for any $\mathscr{C}^2$ function $F$ we have the identity
\be\label{id0}
\frac 1 N \sum_{j=1}^N F (\mu_j) = \frac1{4\pi N} \int \Delta F(z)   \tr   \log (X^* - z^* ) (X-z)    \rdA(z).
\ee
From this formula, it is clear that the small eigenvalues of the Hermitian matrix $(X^* - z^* ) (X-z)  $
play a special role due to the logarithmic singularity at $0$. The key question is to
estimate  the smallest eigenvalues of $(X^* - z^* ) (X-z)$, or in other words, the smallest singular  values
of $ (X-z)$.
 This problem
was not treated  in \cite{Gir1984},
but the  gap was remedied in a series of  papers.
First Bai \cite{Bai1997} was able to treat the  logarithmic  singularity assuming  bounded density and bounded high moments
for the entries of the matrix (see also \cite{BaiSil2006}).
Lower bounds on the smallest singular values were given in Rudelson, Vershynin \cite{Rud2008,RudVer2008}, and subsequently
Tao, Vu \cite{TaoVu2008}, Pan, Zhou \cite{PanZho2010} and G\"otze, Tikhomirov \cite{GotTik2010} weakened the
moments and smoothness assumptions for the circular law, till the optimal $\mbox{L}^2$ assumption,
under which the circular law was proved in \cite{TaoVuKri2010}.

 The purpose of this paper  is to prove a local version of the circular law, up to the optimal scale
$N^{-1/2 + \e}$ (see Section 2 for a precise statement). Below this scale, detailed local statistics will be important
and that  is beyond the scope of the current paper. The main tool of this paper is a detailed analysis
of the self-consistent equations of the  Green functions
$$
G_{ij}(w) = [(X^* - z^* ) (X-z) - w]^{-1}_{ij}.
$$
Our method is related to   the proof of a  local semicircular law in
\cite{ErdYauYin2010Adv}  or to a local   Marchenko-Pastur law  in \cite{PilYin2011}.
We are able
to control $G_{ij}(E + \ii \eta)$ for the energy parameter $E$ in any compact set and sufficient small $\eta$.
This provides  sufficient information to use the formula \eqref{id0} for functions $F$ at the  scales $N^{-1/2+ \e}$.
We also notice that a local Marchenko-Pastur law for $X^*X$ was proved
in \cite{CacMalSch2012}, simultaneously with the present article.

Finally, we remark that the local circular law demonstrates that the eigenvalue distribution in the unit disk is extremely ``uniform''.
If the eigenvalues are distributed in the unit disk by a uniform  statistics or any other statistics with summable  decay of correlations,
then there will be big holes or some clusterings of eigenvalues in the disk.  While the usual circular law does not rule out
these phenomena, the local law established in this paper does. This implies  that  the eigenvalue statistics
cannot be any probability laws  with summable decay of correlations

\nc

\section{The local circular law}

We first  introduce some notations.   Let $X$ be an $N \times N$  matrix with independent centered entries of variance $ N^{-1} $.
The matrix elements can be either real or complex, but for the sake of simplicity we will consider real entries in this paper.
Denote the eigenvalues of $X$ by $\mu_j$, $j=1, \ldots, N$. We will use the following notion of stochastic domination which simplifies
the presentation of the results and their proofs.

\begin{definition}[Stochastic domination]
Let $W=(W_N)_{N\geq 1}$ be family a  random variables and $\Psi=(\Psi_N)_{N\geq 1}$  be  deterministic  parameters.
We say that $W$ is  \emph{stochastically dominated} by  $\Psi$
if for any   $  \sigma> 0$ and $D > 0$ we have
\begin{equation*}
 \P \qB{\absb{W_N} > N^\sigma \Psi_N  } \;\leq\; N^{-D}
\end{equation*}\nc
for sufficiently large $N$.
We denote this stochastic domination property by
\begin{equation*}
W \;\prec\; \Psi\,,\quad or \quad W =\OO_\prec (\Psi).
\end{equation*}
\end{definition}

In this paper, we will assume that  the probability distributions for the   matrix elements
have the uniform subexponential decay property, i.e.,
\be\label{subexp}
\sup_{(i,j)\in\llbracket 1,N\rrbracket^2}\Prob\left(|\sqrt{N}X_{i,j}|>\lambda\right)\leq \vartheta^{-1} e^{-\lambda^\vartheta }
\ee
for some constant $\vartheta >0$ independent of $N$.
This condition can of course be weakened to an hypothesis of boundedness on sufficiently high moments, but the error estimates
in the following Theorem would be weakened as well.
{ We now state our local circular law, which holds up to  the optimal scale $N^{-1/2+\e}$}.

\begin{theorem}\label{lcl}
Let $X$ be an $N \times N$  matrix with independent centered entries of variance $ N^{-1} $.
Suppose that  the probability distributions of the   matrix elements satisfy the
uniformly subexponentially decay condition \eqref{subexp}.
We assume that for some fixed $ \tau>0$, for any $N$ we have $\tau\le ||z_0|-1|\le \tau^{-1} $  ($z_0$ can  depend
on $N$).
Let $f $ be a smooth non-negative function which may depend on $N$,
such that   $\|f\|_\infty\leq C$, $\|f'\|_\infty\leq N^C$ and $f(z)=0$ for $|z|\ge  C$, for some constant $C$ independent of $N$.
Let $f_{z_0}(z)=N^{2a}f(N^{a}(z-z_0))$ be the approximate delta function obtained from rescaling $f$ to the size order $N^{-a}$ around $z_0$.
We denote by $D$ the unit disk. Then for any $a\in(0,1/2]$,
 \be\label{yjgq}
 \left( N^{-1} \sum_{j}f_{z_0} (\mu_j)-\frac1\pi \int_D f_{z_0}(z)    \, \rdA(z)  \right)\prec N^{-1+2a }
 \| \Delta f \|_{L_1}.
 \ee
 \end{theorem}

\section{
Hermitization and local  Green function  estimate}

In the following, we will use the notation
$$
 Y_z=X-z I  \quad
$$
 where $I$ is the identity operator.
Let  $\lambda_j(z)$ be   the $j$-th eigenvalue (in the  increasing ordering) of $Y^*_z  Y_z $.
We will generally omit the $z-$dependence in these notations.
Thanks to the Hermitization technique of Girko \cite{Gir1984},
the first step in proving the local circular law is to understand the local statistics of eigenvalues of $Y^*_z Y_z$,
for $z$ strictly inside the unit circle. In this section, we first recall some well-known facts about the
Stieltjes transform of the empirical measure of eigenvalues of  $Y^*_z Y_z$.
We  then present the key  estimate concerning the Green function of  $Y^*_z Y_z$ in almost optimal spectral windows.
This result will be used later on
to prove a local version of the circular law.

\subsection{Properties of the limiting density of the Hermitization matrix}. Define the Green function
 of $Y^*_z Y_z$  and its trace by
$$
G(w):=G(w,z)=(Y^*_z Y_z-w)^{-1},\quad  m(w):=m(w,z)=\frac{1}{N}\tr G(w,z)=\frac{1}{N}\sum_{j=1}^N\frac{1}{ \lambda_j(z) - w}, \quad w = E + \ii \eta.
$$
We will also need the following version of the Green function later on:
$$
\mG(w):=\mG(w,z)= ( Y_z Y^*_z-w)^{-1}.  \quad
$$
As we will see,  with high  probability
$m(w,z)$ converges to  $m_{\rm c}(w,z)$ pointwise,
as $N\to \infty$ where $ m_{\rm c}(w,z)$  is  the unique
solution of
 \be\label{defmc1}
 m_{\rm c}^{-1}=-w(1+m_{\rm c})+|z|^2(1+m_{\rm c})^{-1}
\ee
with positive  imaginary  part (see Section 3 in \cite{GotTik2010} for the existence and uniqueness of such a solution).
The limit  $ m_{\rm c}(w,z)$ is the Stieltjes transform of a  density $ \rho_{\rm c} (x,z)$ and   we have
$$
 m_{\rm c}(w,z)= \int_\R \frac{\rho_{\rm c} (x,z)}{x-w}\rd x
$$
whenever   $\eta>0$. The function $\rho_{\rm c} (x,z)$ is the limiting eigenvalue density of the matrix $Y^*_z Y_z$
(cf. Lemmas 4.2 and 4.3 in \cite{Bai1997}).
Let
 \be\label{deflampm}
 \lambda_\pm :=\lambda_{\pm }(z):=\frac{( \al\pm3)^3}{8(\al\pm1)} ,\quad \al:=\sqrt{1+8|z|^2}.
\ee
Note  that $\lambda_-$ has the same sign as $|z|-1$.
The  following two propositions  summarize  the properties of $\rho_{\rm c}$ and $m_{\rm c}$ that we will need to understand the main results
in this section. They will be  proved in Appendix
\ref{app:macro}.  In the following, we use the notation $A\sim B$ when $c B \le  A\leq c^{-1}B$,
where $c>0$ is independent of $N$.

\begin{proposition} \label{prorhoc}
The limiting density $\rho_{\rm c}$ is compactly supported and
the following properties regarding $\rho_{\rm c}$  hold.
\begin{enumerate}[(i)]
\item
The support of $\rho_{\rm c}(x, z)$ is $[\max\{0,\lambda_-\}, \lambda_+]$.
\item
As $x\to \lambda_+$ from below,  the behavior of $\rho_{\rm c}(x, z)$ is given by
$\rho_{\rm c}(x, z)\sim \sqrt {\lambda_+-x}.
$
\item For any $\e>0$,
if $ \max\{0,\lambda_-\}+\e\leq x \leq \lambda_+-\e$, then   $\rho_{\rm c}(x, z)\sim 1$.
\item
Near  $\max\{0,\lambda_-\}$, the behavior of $\rho_{\rm c}(x, z)$  can be classified as follows.
   \begin{itemize}
 \item If   $|z|\geq 1+\tau$ for some fixed $\tau>0$, then $\lambda_-> \e(\tau) > 0 $ and  $\rho_{\rm c}(x, z)\sim \1_{x>\lambda_-}\sqrt { x-\lambda_-}$.
\item If  $|z|\leq 1-\tau$ for some fixed $\tau>0$,  then  $\lambda_-< - \e(\tau) < 0$ and  $\rho_{\rm c}(x, z)\sim 1/ \sqrt x $.

\end{itemize}

  All of the estimates in this proposition are uniform in $|z|<1-\tau$,   or $\tau^{-1}\ge |z|\ge 1+\tau$ for fixed $\tau>0$.
 \end{enumerate}
  \end{proposition}

\begin{proposition}\label{tnf}
The preceding Proposition  implies that,  uniformly in $w$ in any compact set,
$$
 |m_{\rm c}(w,z)|=\OO(|w|^{-1/2}  )
$$
 Moreover,
 the following estimates on $m_{\rm c}(w,z)$ hold.
 \begin{itemize}
 \item If   $|z|\geq 1+\tau$ for some fixed $\tau>0$,   then $m_{\rm c}\sim 1$   for $w$ in any compact set.
\item If  $|z|\leq 1-\tau$ for some fixed $\tau>0$,   then $m_{\rm c}\sim |w|^{-1/2} $  for $w$ in any compact set.
\end{itemize}
\end{proposition}

\subsection{Concentration estimate of the Green function up to  the optimal scale. }
We now state  precisely the estimate regarding  the convergence of  $m$ to $m_{\rm c}$.
Since the matrix $Y^*_z Y_z$ is symmetric, we will follow the approach of \cite{ErdYauYin2010Adv}.
We will use extensively the following definition of high probability events.

\begin{definition}[High probability events]\label{def:hp}
 Define
\be\label{phi}
\varphi\;\deq\; (\log N)^{\log\log N}\,.
\ee
Let $\zeta> 0$.
We say that an $N$-dependent event $\Omega$ holds with \emph{$\zeta$-high probability} if there is some constant $C$ such that
$$
\P(\Omega^c) \;\leq\; N^C \exp(-\varphi^\zeta)
$$
for large enough $N$.
\end{definition}

For $  \alpha \geq 0$,  define the $z$-dependent  set
\begin{equation}\label{eqn:Salpha}
\b S(\alpha)\;\deq\; \hb{w \in \C \st  \max (\lambda_-/5, 0) \le   E \leq 5\lambda_+ \,,\;     \varphi^\alpha N^{-1} |m_{\rm c}|^{-1}\leq \eta \leq 10 },
\end{equation}
 where $\varphi$ is defined in \eqref{phi}.  Here  we have  suppressed the explicit $z$-dependence. {  Notice that for  $|z|<1-\e$, as $|m_{\rm c}|\sim|\omega|^{-1/2}$  we allow  $\eta \sim |w| \sim {N^{-2} \varphi^{2\alpha}}$ in the set $\b S(\alpha)$.
This is a key feature of our approach which shows that the Green function estimates hold until a scale much smaller than the typical
$N^{-1}$ value of $\eta$.}

 \begin{theorem}  [Strong local Green function estimates] \label{sempl} Suppose  $\tau\le ||z|-1|\le \tau^{-1} $    for some  $\tau >0$ independent of $N$.
Then  for any $\zeta>0$, there exists $C_\zeta>0$ such that  the following event  holds  \hp{\zeta}:
 \be\label{res:sempl}
\bigcap_{w \in{ {\b {\rm S}}}(C_\zeta)} \hbb{|m(w)-m_{\rm c}(w)|  \leq \varphi^{C_\zeta} \frac{1}{N\eta}}.
 \ee
Moreover,
the individual  matrix elements of
the Green function  satisfy, \hp{\zeta},
\begin{align}\label{Lambdaofinal}
\bigcap_{w \in{ {\b {\rm S}}}(C_\zeta)} \hbb{\max_{ij}\left|G_{ij}-m_{\rm c}\delta_{ij}\right| \leq \varphi^{C_\zeta} \left(\sqrt{\frac{\im\,  m_{\rm c}  }{N\eta}}+ \frac{1}{N\eta}\right)}.
\end{align}
\end{theorem}

\section{ Properties of $\rho_{\rm c}$ and $m_{\rm c}$}\label{sec:pro}
We have introduced some basic properties of $\rho_{\rm c}$ and $m_{\rm c}$ in Proposition \ref{prorhoc} and \ref{tnf}.
In this section, we collect some more useful properties used in this paper,
proved in  Appendix
\ref{app:macro}. Recall that $w = E + \ii \eta$,    $\al=\sqrt{1+8|z|^2}$ from \eqref{deflampm}, and  define  $\kappa:= \kappa (w, z) $ as the distance from $E$ to $\{\lambda_+, \lambda_-\}$:
\be\label{37}
\kappa=\min\{|E-\lambda_-|, |E-\lambda_+|\}.
\ee
For $|z| <  1$, we have $\lambda_- <  0$ { (see Proposition \ref{prorhoc})}, so in this case we define $\kappa:=|E-\lambda_+|$.

\begin{lemma} \label{pmcc1}   There exists $\tau_0>0$ such that for  any $\tau \le \tau_0$ if   $|z|\leq 1-\tau$   and $|w|\leq \tau^{-1} $
then the  following properties concerning  $m_{\rm c}$ hold.  All constants in the following estimates  depend on $\tau$.
\begin{itemize}
\item[Case 1:] $E\geq \lambda_+$ and  $|w-\lambda_+|\ge \tau $. We have
 \be\label{A17}
|\re m_{\rm c}|\sim 1, \quad -\frac{1}{2}\leq  \re m_{\rm c} <0 , \quad \im m_{\rm c}\sim \eta.
 \ee
\item[Case 2:]  $|w-\lambda_+|\le \tau$  (Notice that there is no restriction on whether $E\leq \lambda_+$ or not ). We have
\be\label{A18}
m_{\rm c}(w, z)=- \frac2{3+\al} +  \sqrt{\frac{8(1+\al)^3}{\al(3+\al)^5}}\, (w-\lambda_+ )^{1/2} +\OO(\lambda_+-w), \ee
 and
 \begin{align}\label{esmallfake}
\im m_{\rm c}\sim & \left\{\begin{array}{cc}
 \frac{\eta}{\sqrt{ \kappa}} & \mbox{if\  $\kappa\ge\eta$ and $E\ge \lambda_+$,} \\  & \\
\sqrt{ \eta} & \mbox{if\ $\kappa\le \eta$ or $E\le \lambda_+$.}
\end{array}
\right.
\end{align}
\item[Case 3:] $|w|\le  \tau  $.   We have
\be\label{A20}
 m_{\rm c}(w,z)=\ii\frac{(1-|z|^2)}{\sqrt w} +\frac{1-2|z|^2}{2|z|^2-2}+\OO(\sqrt w)
\ee
as $w\to 0$, and
\be\label{A20a}
\im m_{\rm c}(w,z)\sim |w|^{-1/2}.
\ee
\item[Case 4:]  $|w|\ge  \tau$, $|w-\lambda_+|\ge \tau$  and  $E\leq \lambda_+$.  We have
\be\label{A21}
|m_{\rm c}|\sim 1,\quad \im m_{\rm c}\sim 1.
\ee
\end{itemize}
\end{lemma}

Here Case 1 covers the regime where $E \ge \lambda_+$ and $w$ is far away from $\lambda_+$. Case 2 concerns the regime
that $w$ is near $\lambda_+$, while  Case 3 is for $w$ is near the origin. Finally Case 4 is for $w$ not covered by the first three cases.

 \begin{lemma}\label{pmcc12}  There exists $ \tau_0>0$
such that for  any $\tau \le \tau_0$, if   $|z|\geq 1+\tau$ and $|w|\le \tau^{-1}$ then
the following properties concerning  $m_{\rm c}$ hold.   All constants in the following estimates  depend on $\tau$.
Recall from  \eqref{deflampm}  that
$\lambda_-=\frac{( \al-3)^3}{8(\al-1)} >0$.
\begin{itemize}
\item[Case 1:] $E\geq \lambda_+$ and   $|w-\lambda_+|\ge \tau$. We have
$$
|\re m_{\rm c}|\sim 1,\quad -\frac{1}{2}\leq  \re m_{\rm c}<0 , \quad \im m_{\rm c}\sim \eta.
$$
\item[Case 2:]  $E\le \lambda_-$ and  $|w-\lambda_-|\ge \tau$. We have
$$
|\re m_{\rm c}|\sim 1,\quad 0\leq  \re m_{\rm c}, \quad \im m_{\rm c}\sim \eta.
$$
\item[Case 3:]  $|\kappa+\eta|\le \tau$. We have
$$
m_{\rm c}(w, z)=  \frac2{-3\mp\al} +  \sqrt{\frac{8(\pm1+\al)^3}{\pm\al(\pm3+\al)^5}}\, (w-\lambda_\pm )^{1/2} +\OO(\lambda_\pm-w),
$$
{
\begin{align}
\im m_{\rm c}\sim & \left\{\begin{array}{cc}
 \frac{\eta}{\sqrt{ \kappa}} & \mbox{if  $\kappa\ge\eta$ and $E\notin [\lambda_-, \lambda_+]$,} \\  & \\
\sqrt{\eta} & \mbox{if $\kappa\le \eta$ or $E\in  [\lambda_-, \lambda_+]$.}
\end{array}
\right.
\end{align}
}
\item[Case 4:]  $|w|\ge \tau$,  $|w-\lambda_+|\ge  \tau$  and  $\lambda_- \le E\leq \lambda_+$.  We have
$$
|m_{\rm c}|\sim 1,\quad \im m_{\rm c}\sim 1.
$$
\end{itemize}
  \end{lemma}

Here Case 1 covers the regime $E \ge \lambda_+$ and $w$ is far away from $\lambda_+$. Case 2 concerns the regime
$E \le \lambda_-$ and $w$ is far away from $\lambda_-$. Case 3 is for
 $w$  near  $\lambda_\pm$. Finally Case 4 is for $w$ not covered by the first three cases.

The following lemma concerns the two cases  covered in Lemmas \ref{pmcc1} and \ref{pmcc12}, { i.e.,  $z$ is either strictly inside or outside of the unit disk.}

\begin{lemma}\label{pmcc12.5}   There exists $ \tau_0 >0$ such that for  any $\tau \le \tau_0$   if either the conditions
$|z| \leq 1-\tau$   and $|w|\leq \tau^{-1}$ hold     or   the conditions  $|z| \geq 1+\tau$, $|w|\leq \tau^{-1}$,
$\re\omega\geq \lambda_-/5$ hold,
 then  we have
the following three bounds  concerning  $m_{\rm c}$ (all constants in the following estimates  depend on $\tau$):
\be\label{dA2}
|m_{\rm c}+1|\sim |m_{\rm c}|\sim |w|^{-1/2},
\ee
  \be\label{26ssa}
 \left |  \im \frac{1}{ w(1+m_{\rm c})}  \right |  \leq C  \im  m_{\rm c},   \ee
\be\label{ny27}
 \left|(-1 + |z^2|)
   \left( m_{\rm c}-\frac{-2}{3+\al}\right)  \left( m_{\rm c}-\frac{-2}{3-\al}\right)\right|\geq C\frac{\sqrt{\kappa+\eta}}{ |w|}.
 \ee
\end{lemma}

   \section{Proof of  Theorem \ref{lcl}, local circular law in the bulk}

Our main tool in this section will be Theorem \ref{sempl}, which critically uses the hypothesis $||z|-1|\geq\tau$: when
$z$ is on the unit circle the self-consistent equation (which is a fixed point equation for the function 
$g(m)=(1+w m(1+m)^2)/(|z|^2-1)$ see (\ref{110-31g}) later in this paper) becomes unstable

We follow Girko's idea \cite{Gir1984} of Hermitization, which can be reformulated as the following   identity  (see e.g. \cite{GuiKriZei2009}):
for any smooth $F$
 \be\label{id1}
\frac 1 N \sum_{j=1}^N F (\mu_j) =\frac1{4\pi N} \int \Delta F(z)   \sum_j   \log ( z-\mu_j )(\bar z - \bar \mu_j)   \rdA(z)
=\frac1{4\pi N} \int \Delta F(z)   \tr   \log Y^*_z   Y_z    \rdA(z)
\ee

{ We will use the notation $z= z(\xi)=z_0+N^{-a} \xi$}.
 Choosing $F= f_{z_0}$ defined in Theorem \ref{lcl} and changing the variable to $\xi$,  we can rewrite the identity (\ref{id1}) as
 $$
N^{-1} \sum_j f_{z_0}(\mu_j)
 = \frac1{4\pi} N^{-1+2a}\int (\Delta f)(\xi)   \tr   \log Y^*_z   Y_z    \rdA(\xi)
 = \frac1{4\pi} N^{-1+2a}\int (\Delta f)(\xi)  \sum_j \log \lambda_j(z)   \rdA(\xi).
$$
 Recall that  $\lambda_j(z)$'s are the ordered eigenvalues of $Y_z^*   Y_z $, and define $\gamma_j(z)$
as the classical  location  of $\lambda_j(z)$, i.e.
\be\label{gammadef}
\int_{0 }^ {\gamma_j(z) }\rho_{\rm c}(x,z)\rd x=j/N.
\ee
Suppose we have
\be\label{13}
 \left|\int \Delta f(\xi)  \left( \sum_j \log \lambda_j(z (\xi))-  \sum_j \log \gamma_j(z(\xi))\right)  \rdA(\xi) \right|\prec
 \| \Delta f \|_{L_1}.
  \ee
  Thanks to  Proposition \ref{prorhoc}, one can check  that  uniformly in $ |z| < 1-\tau$, and also in the domain $1+\tau\le |z|\le \tau^{-1}$ ($\tau>0$), for
any $\delta>0$   we have
$$
\left| \sum_j \log \gamma_j(z)-  N \left(\int_0^\infty(\log x)\rho_{\rm c}(x,z) \rd x\right)\right|\leq N^{\delta}
$$
for large enough $N$.
We therefore have
 \begin{align}\label{id22}
N^{-1} \sum_j f_{z_0}(\mu_j)
   = \frac1{4\pi}\int f(\xi)\left(\int_0^\infty(\log x)\Delta_z\rho_{\rm c}(x,z) \rd x\right) \rdA(\xi)
   +\OO_\prec   \| \Delta f \|_{L_1}
\end{align}
where we have used that
$$
\frac1{4\pi} N^{2a}\int \Delta f(\xi) \int_0^\infty(\log x)\rho_{\rm c}(x,z) \rd x    \rdA(\xi)
= \frac1{4\pi}\int f(\xi)\left(\int_0^\infty(\log x)\Delta_z\rho_{\rm c}(x,z) \rd x\right) \rdA(\xi).
$$
 It is known, by  Lemma 4.4 of \cite{Bai1997}, that
 \begin{align}\label{28j}
\int_0^\infty(\log x)\Delta_z\rho_{\rm c}(x,z) dx =4 \chi_D (z).
\end{align}
 Combining \eqref{id22} and \eqref{28j}, we have proved    \eqref{yjgq} provided that we can prove  \eqref{13}.
 To prove \eqref{13}, we need the following rigidity estimate which is a consequence of  Theorem \ref{sempl}.

\begin{lemma}\label{rg} Suppose  $\tau\le ||z|-1|\le \tau^{-1} $    for some  $\tau >0$ independent of $N$.
Then  for any $\zeta>0$, there exists $C_\zeta>0$ such that  the following event  holds  \hp{\zeta}:  for any $ \varphi^{C_\zeta}<j<N-\varphi^{C_\zeta}$ we have
    \be\label{213}
 \gamma _{j-\varphi^{C_\zeta}}\leq \lambda_j\leq \gamma_{j+\varphi^{C_\zeta}}.
  \ee
 and in  the case $|z|\le 1-\tau$,
  \be\label{zth}
\frac{| \lambda_{j }-\gamma_j|}{\gamma_j}\leq   \frac {C \varphi^{C_\zeta}}{ j (1- \frac j N)^{1/3} },
\ee
in  the case $|z|\ge 1+\tau$,
  \be\label{zth2}
\frac{| \lambda_{j }-\gamma_j|}{\gamma_j}\leq   \frac {C \varphi^{C_\zeta}}{ (\min\{\frac{j}{N}, 1-\frac{j}{N}\})^{1/3}N }.
\ee
    \end{lemma}

\begin{proof} First,  with \eqref{res:sempl} and the definition (\ref{eqn:Salpha}), for any $\zeta$ there exists $C_\zeta>0$ such that
\be
\max_{E+i\eta\in {\b {\rm S}}(C_\zeta)}  { \eta |   m(E+\ii\eta)-  m_{\rm c}(E+\ii\eta)|\le  C\varphi^{2C_\zeta}N^{-1}}.
\label{ym}
\ee
holds with \hp{\zeta}.  It also implies that for $\eta=\varphi^{C_\zeta}N^{-1}|m_{\rm c} |^{-1}$,
 \be\label{ym1}
\eta  \im  m(E+\ii\eta)
\le  C\varphi^{2C_\zeta}N^{-1}.\quad \ee
Then using  the fact that
$
\eta  \im  m(E+\ii\eta)
$ and $
\eta\im m_{\rm c}(E+\ii\eta)
$ are increasing with $\eta$, we obtain that \eqref{ym1} holds for   any $0\leq \eta\leq \OO( \varphi^{C_\zeta}N^{-1}|m_{\rm c} |^{-1})$
\hp{\zeta}. Notice that $\im m$ and $\im m_{\rm c}$ are positive number. Define the interval
$$I_E=[E_1,E_2]=[\gamma_j, 4\lambda_+]$$ and define $\eta_j\geq 0$ as the smallest positive solution of
$$
\eta_{j} =2\varphi^{C_\zeta}|m_{\rm c}(E_j+ \ii \eta_j)|^{-1}N^{-1},\quad j=1, \;2.
$$
Since $$ \#\{j: E-\eta\leq \lambda_j\leq E+\eta\}\leq CN\eta\im m(E+\ii\eta ),$$
we have by  \eqref{ym1} that
\be\label{dbu}
 \#\{j: E_1-\eta_1\leq \lambda_j\leq E_1+\eta_1\}
 +
 \#\{j: E_2-\eta_2\leq \lambda_j\leq E_2+\eta_2\}\leq C\varphi^{2C_\zeta}.
\ee

Using the Helffer-Sj\"ostrand functional calculus (see e.g. \cite{Dav1995}),  letting $\chi(\eta)$ be a smooth cutoff function  with  support in $[-1,1]$,  with $\chi(\eta)= 1$ for $ |\eta|\leq1/2$  and with bouded derivatives, we have  for any $q: \R\to \R$,
$$
q(\lambda)=\frac{1}{2\pi}\int_{\R^2}\frac{\ii y q''(x)\chi(y)+\ii (q(x)+\ii y q'(x))\chi'(y)}{\lambda-x-\ii y}\rd x\rd y.
$$
To prove \eqref{213}, we  choose $q$ to be  supported in  $[E_1 , E_2 ]$ such that $q(x)=1$ if $x\in [E_1+\eta_1, E_2-\eta_2]$ and $|q'|\leq C(\eta_{i})^ {-1}$, $|q''|\leq  C(\eta_{i})^{-2}$ if $|x-E_i|\leq \eta_i$.
We now claim that
\be\label{ys0}
\left|\int q(\lambda)\Delta\rho(\lambda)\rd\lambda \right|\leq C\varphi^{2C_\zeta}N^{-1},\   {\rm where}\ \Delta\rho=\rho-\rho_{\rm c},\ \rho=\frac{1}{N}\sum_j \delta_{\lambda_j(z)}.
\ee
Combining \eqref{ys0} and \eqref{dbu}, we have for any $1\leq j\leq N$,
$$
\#\{k: \lambda_k\ge \gamma_j\}-(N-j)=\OO(\varphi^{2C_\zeta})
$$
 which implies \eqref{213} with $C_\zeta$ in \eqref{213}  replaced by $2 C_\zeta$.

It   remains to prove \eqref{ys0}.
 Since $q$ and $\chi$ are real,  with $\Delta m=m-m_{\rm c}$
\begin{align}
\left|\int q(\lambda)\Delta\rho(\lambda)\rd\lambda \right|
\leq & C\int_{\R^2} \big( |q(E)| +|\eta| |q'(E)|\big) |\chi'(\eta)|
| \Delta m(E+\ii\eta)| \rd E\rd \eta\nonumber \\
& +C\sum_i\left|\int_{|\eta|\leq \eta_i}\int_{|E-E_i|\leq \eta_i} \eta q''(E) \chi(\eta)
\im \Delta m(E+\ii\eta)\rd E\rd \eta\right|\nonumber\\
&+
C\sum_i\left|\int_{|\eta|\geq \eta_i}\int_{|E-E_i|\leq \eta_i} \eta q''(E)\chi(\eta) \im\Delta m(E+\ii\eta)\rd E
\rd \eta\right|,\label{intr2fe1}
\end{align}
The first term is estimated by
\be\label{ys1}
 \int_{\R^2} ( |q(E)| +|\eta| |q'(E)|) |\chi'(\eta)|
| \Delta m(E+\ii\eta)| \rd E\rd \eta
 \le CN^{-1}\varphi^{C_\zeta},
\ee
 using \eqref{res:sempl} and that on the support of $\chi'$ is in $1\ge |\eta|\ge 1/2$.

For the second term in the r.h.s. of \eqref{intr2fe1}, with  $|q''|\leq C\eta_i^{-2}$,
\eqref{ym} and \eqref{ym1}, we obtain
\be
  \mbox{second term in r.h.s. of \eqref{intr2fe1}}
\le  CN^{-1}\varphi^{C_\zeta}. \label{ys2}
\ee
{ We now integrate the third term in
 \eqref{intr2fe1} by parts first in $E$, then in $\eta$
(and use the Cauchy-Riemann equation $\frac{\partial}{\partial E}\im (\Delta m)=-\frac{\partial}{\partial \eta}
\re(\Delta m)$) so that
 \begin{align*}
\int \eta q''(E) \chi(\eta)
\im (\Delta m(E+\ii\eta))\rd E\rd \eta=
 &-\int_{|E-E_i|\leq \eta_i} \eta_i \chi(\eta)
  q' (E)\re (\Delta m(E+\ii\eta))   \rd E
 \\
&- \int (\eta\chi'(\eta)+\chi(\eta))  q'(E)\re (\Delta m(E+\ii \eta))\rd E\rd \eta
\end{align*}

 We therefore can bound the third term in
 \eqref{intr2fe1} with absolute value by}
\begin{align}\label{temp7.501}
&C\sum_i\int_{ |E-E_i|\leq \eta_i}\eta_i |q'(E)| |\re{ \Delta} m(E+\ii\eta_i)|\rd E
\\\nonumber
+&
C\sum_i \eta_i^{-1}\int_{\eta_i\le \eta\leq 1}\int_{ |E-E_i|\leq \eta_i} |\re { \Delta} m(E+\ii\eta)|\rd E\rd \eta
+\int_{\R^2}   |\eta| |q'(E)|  |\chi'(\eta)|
| \Delta m(E+\ii\eta)| \rd E\rd \eta
\end{align}
where the last term can be bounded as the first term in r.h.s. of  \eqref{intr2fe1}. By using \eqref{ym}
  we have
\begin{align}\nonumber
\eqref{temp7.501}\leq& CN^{-1}\varphi^{C_\zeta}+
CN^{-1}\varphi^{C_\zeta}\sum_i\eta_i^{-1} \int_{|E-E_i|\leq \eta_i}\rd E
 \int_{\eta_i\le \eta\le 1}\frac{1}{\eta N}\rd \eta
\leq  CN^{-1}\varphi^{C_\zeta+1} \end{align}
where we used $\eta_i\ge N^{-C}$. Together with \eqref{ys1} and \eqref{ys2}, we obtain \eqref{ys0} and complete the proof of \eqref{213}.

Now we prove \eqref{zth}. Using  \eqref{gammadef} and Proposition \ref{prorhoc}, we have
\be\label{5188}
 \gamma_j=\OO( j^2N^{-2}), \quad j \le N/2; \qquad  \gamma_j=\lambda_+ -\OO\left ( \frac { N-j} { N} \right )^{2/3}, \quad j \ge N/2.
 \ee
One can check easily that
 $$
 \gamma_j- \gamma_{j-1}=\OO \left ( \frac j {N^{5/3}(N-j)^{1/3}} \right )
 $$
 and    for $j\geq 2$
  \be\label{214}
\frac{| \gamma_j-\gamma_{j\pm1}|}{\gamma_j}\leq C j^{-1}N^{1/3}(N-j)^{-1/3}\leq   \frac {C \varphi^{C_\zeta} }{ j (1-\frac{j}{N})^{1/3} }.
\ee
 Combining \eqref{214} with  \eqref{213}, we obtain \eqref{zth}.

 For \eqref{zth2}, the proof is similar to the above reasoning,
but simpler: in this case $\gamma_j\sim 1$ for $j\leq N/2$. For $j\geq N/2$, $\gamma_j$ is bounded as \eqref{5188}, and
one can check if  $1+\tau\leq |z|\le \tau^{-1}$, Proposition \ref{prorhoc},  we have
$$
 \gamma_j- \gamma_{j-1}=\OO \left (  \left(\min\left\{\frac{j}{N}, 1-\frac{j}{N}\right\}\right)^{-1/3}N^{-1}  \right )
 $$
which implies \eqref{zth2}.
\end{proof}

We return to  the proof of the local circular law,  Theorem  \ref{lcl}.
We now only need to prove  \eqref{13} from Lemma \ref{rg}.
From \eqref{zth} and \eqref{zth2}, we have
$$
 \left| \log \lambda_j(z)-\log \gamma_j(z)\right|\leq C \frac{| \lambda_{j }-\gamma_j|}{\gamma_j}\leq \frac {C \varphi^{C_\zeta}}{ j (1- \frac j N)^{1/3} }, \quad |z|\leq 1-\tau
$$
and
$$
 \left| \log \lambda_j(z)-\log \gamma_j(z)\right|\leq C \frac{| \lambda_{j }-\gamma_j|}{\gamma_j}\leq \frac {C \varphi^{C_\zeta}}{ (\min\{\frac{j}{N}, 1-\frac{j}{N}\})^{1/3}N }, \quad 1+\tau\le |z|\leq \tau^{-1}.
$$
Notice that, for large enough $C$, there is a constant $c>0$ such that  for any $j$ we have
$$
\lambda_j\leq N^C
$$
with probability larger than $1-\exp({{ -N^c}})$
{ (for this elementary fact, one can for example see that the
entries of $X$ are smaller that $1$ with probability greater than $1-\vartheta^{-1}e^{-N^\vartheta}$ by the subexponential
decay assumption (\ref{subexp})
and then use $\sum \lambda_j=\tr Y^* Y $)},
so together with the above bounds on $\left| \log \lambda_j(z)-\log \gamma_j(z)\right|$ this proves that
for any $\zeta>0$, there exists $C_\zeta>0$ such that
\be\label{ydy}
\left|    \sum_{j> \varphi^{C_\zeta}}
 \left( \log \lambda_j(z)-\log \gamma_j(z)\right)\right|
 \leq  \varphi^{2C_\zeta}
\ee
 \hp{\zeta}.
 Furthermore, one can see that or estimates hold uniformly for $z$'s in this region.

 On the other hand,  the following important Lemma \ref{lem:syc0} holds, concerning the smallest eigenvalue. It implies that
$$
\sum_{j\leq \varphi^{C_\zeta}}  | \log \lambda_j(z) |  \prec 1
$$
holds uniformly for $z$
in any fixed compact set. It is easy to check that for any $\delta>0$, for large enough $N$,
$$
\sum_{j\leq \varphi^{C_\zeta}}  | \log \gamma_j(z) |  \leq N^{  \delta}.
$$
Hence we can extend the summation in \eqref{ydy} to all $j \ge 1$, which gives \eqref{13} and     completes the proof of Theorem  \ref{lcl}.

   \begin{lemma}  [Lower bound on the smallest eigenvalue] \label{lem:syc0}
Under the same assumptions of Theorem \ref{lcl},
$$
 | \log \lambda_1(z) |  \prec 1
$$
holds uniformly for $z$
in any fixed compact set.
   \end{lemma}

\begin{proof}
This lemma follows\footnote{Strictly speaking, this bound was proved for identically distributed entries,
but the proof extends to the case of distinct distributions, provided that, for example, a uniform subexponential decay holds.} from
\cite{RudVer2008} or  Theorem 2.1 of \cite{TaoVu2008}, which gives the required estimate uniformly in $z$. Note that
the  typical size of $\lambda_1$ is $N^{-2}$ \cite{RudVer2008}, and
we need a much weaker bound of type $\Prob(\lambda_1(z)\leq e^{-N^{-\e}})\leq N^{-C}$ for any $\e,C>0$.
This estimate is very  simple to prove if, for example, the entries of $X$ have a density bounded by $N^C$.
Then, from the variational characterization $\lambda_1(z)=\min_{|u|=1}\|X(z)u\|^2$, one easily gets
$$\lambda_1(z)^{1/2}\geq N^{-1/2}\min_{k\in\llbracket 1,N\rrbracket}\mbox{dist}(X(z)e_k,\mbox{span}\{X(z)e_\ell,\ell\neq k\})
=
N^{-1/2}\min_{k\in\llbracket 1,N\rrbracket}|\langle X(z)e_k, u_k(z)\rangle|,$$
where  $u_k(z)$ is a unit vector independent of $X(z)e_k$. By conditioning on $u_k(z)$,  the
result of this lemma is straightforward  since  the matrix entries have a density.
 \end{proof}

 \section{
 Weak local Green function estimate}\label{sec:PTs}

 In this section, we make a first step towards Theorem \ref{sempl}, with a weaker version of it, stated hereafter.

  \begin{theorem} [Weak  local Green function estimates]  \label{wempl}
Under the assumption of Theorem \ref{sempl}, the following event  hold  \hp{\zeta} (see (\ref{eqn:Salpha}) for the definition of $\b {\rm S}$):
 \be\label{res:wempl}
 \bigcap_{w \in{ \b {\rm S}}(b )} \hbb{\max_{ij}|G_{ij}(w)-m_{\rm c}(w)\delta_{ij}|  \leq \varphi^{C_\zeta}
 \frac {1}{|  w^{1/2}|}
\left(
\frac{|  w^{1/2}|}{N\eta}
\right)^{1/4} }, \quad b > 5 C_\zeta.   \ee
  \end{theorem}

This theorem will be proved in the subsequent subsections.

 \subsection{Identities for Green functions and their minors. }
There are many different ways to form minors for the matrices  $Y^* Y$ and $ Y Y^* $. We will use the following
definition (where we use the notation $\llbracket a, b\rrbracket=[a,b]\cap\mathbb{Z}$).

\begin{definition} \label{definition of minor}
Let $\bb T, \bb U \subset \llbracket 1, N\rrbracket$. Then we define $Y^{(\bb T, \bb U)}$ as the $ (N-|\bb U|)\times (  N-|\bb T|)  $
matrix obtained by removing all columns of $Y$ indexed by $i \in \bb T$ and all rows  of $Y$ indexed by $i \in \bb U$. Notice that we keep the labels of
indices of $Y$ when defining $Y^{(\bb T, \bb U)}$.

Let $\by_i$ be the $i $-th column of $Y$ and  $\by^{(\bb S)}_i$ be the vector obtained by removing $\by_i (j) $ for
 all  $ j \in  \bb S$. Similarly we define $\mathrm y_i$ be the $i $-th row of $Y$.
Define
\begin{align*}
G^{(\bb T, \bb U)}=\Big [  (Y^{(\bb T, \bb U)})^* Y^{(\bb T, \bb U)}- w\Big]^{-1},\ \  & m_G^{(\bb T, \bb U)} =\frac{1}{N}\tr G^{(\bb T, \bb U)},
\\
\mG^{(\bb T, \bb U)}= \Big [ Y^{(\bb T, \bb U)}(Y^{(\bb T, \bb U)})^*- w \Big ]^{-1},\ \ & m_\mG^{(\bb T, \bb U)} =\frac1N\tr \mG^{(\bb T, \bb U)}.
\end{align*}
By definition,  $m^{(\emptyset, \emptyset)} = m$.
 Since the eigenvalues of $Y^* Y $ and $Y Y^*$ are the same except the zero eigenvalue, it is easy to check that
\be\label{35bd}
m_G^{(\bb T, \bb U)}(w) =m_\mG^{(\bb T, \bb U)} +\frac{|\bb U|-|\bb T|}{N  w}
\ee
For $|\bb U|=| \bb T|$, we define
\be\label{d trGmG}
m ^{(\bb T, \bb U)}:= m_G^{(\bb T, \bb U)} = m_\mG^{(\bb T, \bb U)}
\ee
\end{definition}

By definition, $G^{(\bb T, \bb U)} $  is a $(N-|\bb T|)\times (N-|\bb T|)$ matrix and $\mG^{(\bb T, \bb U)} $  is a $(N-|\bb U|)\times (N-|\bb U|)$ matrix.
For  $i$ or $j\in\bb T$, $G_{ij}^{(\bb T, \bb U)}$ has no meaning from the previous definition.  But we define  $G_{ij}^{(\bb T, \bb U)} = 0$  whenever
either $i$ or $j \in \bb T$.   Similar convention applies to $\mG_{i j}^{(\bb T, \bb U)}$, which is zero if $i$ or $j \in \bb U$.

 Notice that we can view  $Y_z Y^*_z = (W_{z^*})^* W_{z^*} $ where $ W_{z^*} = Y^*_z$, so all properties of  $G^{(\bb T, \bb U)}$ have parallel versions
for  $\mG^{(\bb U, \bb T)}$. We shall call this property   {\it row-column  reflection symmetry}, i.e., we interchange
 $G^{(\bb U,\bb T)}, Y, z, \by_i $ by   $\mG^{(\bb T, \bb U)}, Y^*, z^*, \mathrm y_i $. Here
$\by_i$ is a $N\times 1 $ column vector and $\mathrm y_i$ a $1\times N $ row vector.
The following lemma provides the formulas relating Green functions and their minors.

  \begin{lemma} [Relation between $G$, $G^{(\bb T,\emptyset)}$ and  $G^{( \emptyset, \bb T)}$]  \label{lem: GmG}
 For $i,j \neq k  $ ( $i = j$ is allowed) we have
\be\label{111}
 G_{ij}^{(k,\emptyset)}=G_{ij}-\frac{G_{ik}G_{kj}}{G_{kk}}
,\quad
\mG_{ij}^{(\emptyset,k)}=\mG_{ij}-\frac{\mG_{ik}\mG_{kj}}{\mG_{kk}},
\ee
\be\label{Gik}
 G^{ ( \emptyset,i)}
= G+\frac{(G  {\mathrm y} _i^*) \, ( {\mathrm y} _i  G)}
{1-  {\mathrm y} _i G    {\mathrm y} _i ^*}
,\quad
G
 =G^{ ( \emptyset,i)}-\frac{( G^{ ( \emptyset,i)} {\mathrm y} _i^*)  \,
 (  {\mathrm y} _i  G^{ ( \emptyset, i)})}
 {1+   {\mathrm y} _i  G^{ ( \emptyset,i)} {\mathrm y} _i ^* },
 \ee
and
$$
\mG^{ (i,\emptyset)}
=\mG+\frac{(\mG  \by_i) \, (\by_i^* \mG)}
{1-   \by_i^*   \mG     \by_i  }
,\quad
 \mG
 =\mG^{ (i,\emptyset)}-\frac{(\mG^{ (i,\emptyset)}  \by_i)  \,
 ( { \by_i}^ * \mG^{ (i,\emptyset)})}
 {1+  \by_i^*\mG^{ (i,\emptyset)}   \by_i  }.
$$

Furthermore, the following crude  bound on the difference between $m$ and $m_G^{(\mathbb{U}, \mathbb{T})}$ holds:
for $\bb U,  \bb T\subset \llbracket 1, N\rrbracket$    we   have
\be\label{37mk}
|m-m^{(\bb U, \bb T)}_G|+|m-m^{(\bb U, \bb T)}_\mG| \leq   \frac{|\bb U|+|\bb T|}{N\eta}.\quad
 \ee
   \end{lemma}

 \begin{proof} By the   row-column  reflection symmetry, we only need to prove those formulas involving  $G$.  We first prove \eqref{111}.  In \cite{ErdYauYin2010PTRF}-\cite{ErdYauYin2010Adv}, was proved a lemma concerning  Green functions of matrices  and their  minors.
This lemma is stated as Lemma \ref{basicIG} in Appendix B.  Let
  \be\label{49H}
  H:= Y^* Y
  \ee
For $\bT\subset \llbracket 1, N\rrbracket$, denote  $H^{[\bT]}$ as the $N-|\T|$ by $N-|\T|$ minor of $H$ after removing the
 $i$-th   rows and columns index by $i\in \T$. Following the convention in Definition \ref{basicd}, we define
 \be\label{410G}
 G^{[\bT]}=(H^{[\bT]}-wI)^{-1}.
 \ee
 By definition, we have
 \be\label{mfn}
 G^{[\bT]}=G^{(\bT,\emptyset)}.
 \ee
 Then we  can apply   \eqref{3} to  $G^{(\bT,\emptyset)}$ and obtain \eqref{111}.

 We now prove \eqref{Gik}.   Recall the rank one perturbation formula
$$
(A +   \bv^* \bv)^{-1} = A^{-1} - \frac {  (A^{-1}  \bv^*)  (\bv A^{-1})} { 1 +   \bv A^{-1}   \bv^*}
$$
where $\bv$ is a row vector and $\bv^*$ is its Hermitian conjugate.
Together with
$$
G^{-1}=Y^* Y-w I= \sum_j {\mathrm y}_j^* {\mathrm y}_j-w I
=\left(G^{(\emptyset, i)}\right)^{-1}+{\mathrm y}_i^* {\mathrm y}_i
$$
we obtain \eqref{Gik}.

We now prove \eqref{37mk}.  With \eqref{111}, we have
$$
m_{G}^{(i,\emptyset)}-m=-\frac{1}{N}\frac{\sum_{j}G_{ji}G_{ij}}{G_{ii}}.
$$
Moreover, by diagonalization in an orthonormal basis and the obvious identity $|(\lambda-\omega)^{-2}|=\eta^{-1}\im[(\lambda-\omega)^{-1}]$ ($\lambda\in\mathbb{R}$), we have
$$
\left|\sum_{j}G_{ji}G_{ij}\right|=|[G^2]_{ii}|=\frac{\im G_{ii}}{\eta},
$$
so we have proved  that
\be\label{b1}
|m-m^{(i, \emptyset)}_G| \leq    \frac{1 }{N\eta}.
\ee
By \eqref{d trGmG},   \eqref{b1} holds for $m_\mG^{( i, \emptyset)}$ as well.  Similar arguments can be used to prove \eqref{37mk} for $m_
 \mG^{(i, j)}$, $ m_ G^{(i, j)}$ and the general cases. This completes the proof of Lemma \ref{lem: GmG}.
\end{proof}

The next step is to derive  equations between the matrix and its minors.  The main results are stated as
the following Lemma \ref{idm}. We first need the following definition.

 \begin{definition}\label{Zi-def}
 In the following, $\E_X$ means the integration with respect to the random variable $X$.
For any $\bb T\subset \llbracket 1, N\rrbracket$, we introduce the   notations
$$
Z^{(\bb T)}_{i }:=(1-\E_{{\mathrm y}_i})
 {\mathrm y}^{(\bb T)}_i  G^{(\bb T, i)} {\mathrm y}_i^{(\bb T)*}
$$
and
$$
\cal Z^{(\bb T)}_{i }:=(1-\E_{\by_i})
\by_i^{(\bb T) *} \mG^{(i, \bb T)} \by_i^{(\bb T)}.
$$
 Recall by our convention that
$\by_i$ is a $N\times 1 $ column vector and $\mathrm y_i$ is a $1\times N $ row vector.
For simplicity we will write
$$
Z _{i }
=Z^ {(\emptyset)}_{i}, \quad \cal Z _{i }
=\cal Z^ {(\emptyset)}_{i}.
$$
\end{definition}

 \begin{lemma}  [Identities for  $G$, $\mG$, $Z$ and  $\cal Z$]   \label{idm}
 For any $ \T\subset \llbracket 1, N\rrbracket$, we have
\begin{align}\label{110}
 G^{(\emptyset , \bb T)} _{ii}
 & =   - w^{-1}\left[1+  m_\mG^{(i, \bb T)}+   |z|^2 \mG_{ii}^{(i, \bb T)} +\cal Z^{(\bb T)}_{i } \right]^{-1},
\end{align}
\be\label{110b}
 {G_{ij} ^{(\emptyset , \bb T) } }
  =   -wG_{ii}^{(\emptyset , \bb T) } G^{(i,\bb T)}_{jj}
\left(   \by_i^{(\bb T)*}  \mG^{(ij, \bb  T)}  \by_j^{(\bb T)}\right) , \quad i\neq j,
\ee
where, by definition,  $\mG_{ii}^{(i,\bb T)}=0$ if $i\in \bb T$.  Similar results hold for $\mG$:
\be\label{110c}
\left[\mG^{(\bb T, \emptyset)} _{ii} \right]^{-1}
  =   - w\left[1+  m_ G^{(\bb  T,i)}+   |z|^2  G_{ii}^{(\bb T,i)} + Z^{(\bb T)}_{i } \right]
\ee
\be\label{110d}
 {\mG_{ij}^{(\bb T, \emptyset)}}
  =   -w\mG_{ii}^{(\bb T, \emptyset)}\mG^{(\bb T, i)}_{jj}\left( \mathrm y_i^{(\bb T)} G^{( \bb T,ij)}   \mathrm y_j^{(\bb T)*}\right), \quad i\neq j.
\ee
\end{lemma}

\begin{proof}  By the row-column  reflection symmetry, we only need to prove the $G$ part of this lemma. Furthermore, for simplicity,  we prove the case $T=\emptyset$, the  general case can be proved in the same way.

We first prove \eqref{110}. Let $H=Y^* Y$. Similarly to \eqref{49H} and \eqref{410G}, we define $G^{[i]}$ and $H^{[i]}$. Then using \eqref{1} and \eqref{mfn}, we have
$$
\left[G  _{ii} \right]^{-1}=h_{ii}-w-\sum_{k,l\neq i}h  _{ik}G^{(i, \emptyset )}_{kl}h_{li}.
$$
From the definition of $ H $, we have $h  _{ik}= \by_i^* \by_k  $. Then
\be\label{422l}
\left[G  _{ii} \right]^{-1}= \by_i^* \by_i -w-
  \by_i^* Y^{(i, \emptyset )}G^{(i, \emptyset )} \left(Y^{(i, \emptyset )}\right)^* \by_i.
\ee
 For any matrix $A$, we have the identity
\be\label{499}
A (A^*  A - w)^{-1}  A^*=1+ w ( A  A^*   - w)^{-1},
\ee
and as a consequence
\be\label{424l}
Y^{(i, \emptyset )}G^{(i, \emptyset )} \left(Y^{(i, \emptyset )}\right)^*= 1+ w\mG^{(i, \emptyset )}.
\ee
Combining \eqref{422l} and \eqref{424l}, we have
\be\label{5.888}
\left[G  _{ii} \right]^{-1}=-w-w\,\by_i^* \mG^{(i, \emptyset )} \by_i
\ee
We now write
$$
\by_i^* \mG^{(i, \emptyset )} \by_i =
 \E_{\by_i}\by_i^* \mG^{(i, \emptyset )} \by_i
 +\cal Z_i
$$
By definition
$$
\E_{\by_i}\by_i^* \mG^{(i, \emptyset )} \by_i
  =\frac1N\tr  \mG^{(i, \emptyset)}+|z|^2\mG^{(i, \emptyset)} _{ii}=m_\mG^{(i, \emptyset)}+|z|^2 \mG^{(i, \emptyset)} _{ii}
$$
 which complete the proof of \eqref{110}.

We now prove \eqref{110b}.  As above, using now \eqref{2},  we have
$$
 {G_{ij} ^{(\emptyset ,\bb T) } }
  =    G_{ii}^{(\emptyset ,\bb T) } G^{(i,\bb T)}_{jj}
  \left(h_{ij}-\sum_{kl\neq ij}h  _{ik}G^{(ij, \emptyset )}_{kl}h_{lj} \right)
$$
where
$$
h_{ij}-\sum_{kl\neq ij}h  _{ik}G^{(ij, \emptyset )}_{kl}h_{lj} =
 \by_i^* \by_j
- \by_i^* Y^{(ij, \emptyset )}G^{(ij, \emptyset )} \left(Y^{(ij, \emptyset )}\right)^* \by_j.
$$
 Then using \eqref{499} again, we obtain \eqref{110b}.
\end{proof}

\subsection{ The self-consistent equation and its stability. }
We now derive the self-consistent equation for $m(w)$ and its stability estimates.
 Following \cite{ErdYauYin2010Adv}, we introduce the following control parameter:

\begin{definition} Define the control parameter
$$
\Psi= \left(\sqrt{\frac{\im m_{\rm c}+\Lambda}{N\eta}}+\frac{1}{N\eta}\right), \quad \Lambda = |m-m_{\rm c}|
$$
Notice that all quantities depend on $w$ and $z$.
Furthermore,   if $\Lambda \le C |m_c|$ then for  $w \in\b {\rm S}(b)$ (see (\ref{eqn:Salpha})),
\be\label{Psib}
|m_{\rm c}|^{-1} \Psi\le \frac{ 1}{\sqrt{  N\eta  |m_c|}}+\frac{1}{N\eta |m_c| } \le
 C  \varphi^{-b/2}.
\ee
The quantity $|m_{\rm c}|^{-1} \Psi$ will be our controlling small parameter in this paper.   \nc
\end{definition}

{

Before we start to prove Theorem \ref{sempl}, we make the following observation. The parameter $z$ can be either inside
the unit ball or outside of it.  Recall the  properties of $m_{\rm c}$ in section \ref{sec:pro}.
 By  Lemma \ref{prorhoc}, the limiting density $\rho_c$ of $YY^*$ is  supported on $[\lambda_-, \lambda_+]$, where $\lambda_- < 0$ and $\lambda_+ \sim 1$ when $|z| \le 1 - \tau$. Since $\lambda_- < 0$ in this case, we will never approach $\lambda_-$.
On the other hand, we will have to consider the behavior when $w \sim 0$.
  When $ 1+ \tau \le |z| \le \tau^{-1}$, we have $\lambda_-> 0$ and $w$ stays away from the origin by definition of $\b {\rm S} (C_\zeta)$, i.e.,
  the condition $E \ge \lambda_-/5$.
Our approach to the local Green function estimates
will use the self-consistent equation of $m(w)$. This approach depends crucially on the stability properties of this equation which can be
divided roughly into three cases: $w$ near the edges $\lambda_\pm$, $w \sim 0$ or $w$ in the bulk (defined here as the
rest of possible $w \in \b {\rm S} (C_\zeta)$).  From Lemma  \ref{pmcc1} and  Lemma \ref{pmcc12}, the behavior of $m_c$  near the edges $\lambda_\pm$
when $ |z| \ge 1 + \tau$ are identical to its behavior near the edge $\lambda_+$
when $ |z| \le 1 - \tau$. In the bulk, the behavior for both cases are the same. Thus we will only consider the case $ |z| \le 1 - \tau$ since
it covers all three different behaviors. Hence from now on, we will assume that $|z| \le 1 -\tau$.
 We emphasize that $\im m_{\rm c} \ll |m_{\rm c}|$  when   $|\lambda_+-w|\ll 1$. }
All stability results concerning  the self-consistent equation will be under the following assumption \eqref{s3}.

\begin{lemma} [Self consistent equation]\label{cor:self}
Suppose $|z| \le 1 -\tau$ for some $\tau > 0$. Then there exists  a small constant $ \alpha > 0$  independent of $N$  such that
if the estimate
 \be\label{s3}
\Lambda  \le  \alpha   |m_{c}|
\ee
holds    for some $|w|\le C$ on a set $A$ in the probability space of matrix elements for $X$,
then  in the set $A$ we have  \hp{\zeta}
\begin{align}\label{110-31g}
    w\, m   (1+ m)^2  - m  |z|^2           + 1 + m   =  \Upsilon,
\quad   \Upsilon = \OO\left ( \varphi^{   Q_\zeta} \Psi  \right )\; ,
\end{align}
 provided that $ w \in\b {\rm S}(b)$ for some   $b > 5 Q_\zeta $ with $Q_\zeta$ defined in Lemma \ref{lem:bh}. \nc
\end{lemma}

\begin{proof}
By  \eqref{dA2}, \eqref{26ssa} and  \eqref{s3}, for $|z| \le  1-t$
the following inequalities   hold on the set $A$:
 \be\label{81}
|w|^{-1}\frac 1 { |1+  m|^2  } \le |w|^{-1}\frac 1 { |1+  m_{\rm c} + \OO(\Lambda)  |^2   }  \le C,
\ee
\be\label{80}
\left |  \im  \frac 1 { w(1+  m)   }   \right | \le   \left | \im  \frac 1 { w(1+  m_{\rm c})   }  \right | + \left |  \frac 1 { w(1+  m_{\rm c})   }  (m-m_{\rm c}) \frac 1 { (1+  m)   }
 \right | \le \im m_{\rm c} + C \Lambda.
\ee
Furthermore, using \eqref{81}, \eqref{dA2}, \eqref{26ssa}, \eqref{s3} and (\ref{defmc1}), we have in the set $A$
\be\label{82}
  1+ m  - \frac {|z|^2}   { w(1+  m)   }   = 1+ m_{\rm c}  -  \frac {|z|^2}   { w(1+  m_{\rm c})   } + \OO (\Lambda )   = \frac 1 { w m_{\rm c}} + \OO (\Lambda ).
 \ee
 The origin of the self-consistent equation (\ref{110-31g}) relies on the choice
$\bT = \{i\}$ in  \eqref{110c}:
\be\label{110c1}
\left[\mG^{(i, \emptyset)} _{ii} \right]^{-1}
  =   - w\left[1+  m_ G^{( i ,i)}+Z^{(i)}_{i }   \right].
\ee
   By  definition of $\Psi$ and \eqref{37mk},
\be\label{43}
 |m_ G^{( i ,i)}- m| \le \frac C {  N \eta} \le  C \Psi .
\ee
Moreover, we have from \eqref{130} that \hp{\zeta} in $A$
\be\label{44}
| Z^{(i)}_{i }| \le    \varphi^{ Q_\zeta/2} \sqrt{\frac{\im m_G^{(i,i)}+ |z|^2 \im G^{(i, i)}_{ii} }{N\eta}} \le
  \varphi^{ Q_\zeta/2} \Psi
\ee
where we have used \eqref{43}, \eqref{s3}  and, by definition,  $G^{(i, i)}_{ii} = 0$.
We would like to  estimate $ (\mG_{ii}^{(i,\emptyset)})^{-1}$ in (\ref{110c1}) by treating $( 1+  m  )$ as the main term and the rest as error terms.
 From the  equations \eqref{s3}  and  \eqref{Psib},
the ratio between the error terms and the main term  for  $ w \in\b {\rm S}(b)$ with  $b > 5 Q_\zeta $ is bounded by
\be
 |m|^{-1} | Z^{(i)}_{i }| +  |m|^{-1}  |m_ G^{( i ,i)}- m|  \le \varphi^{- Q_\zeta}.
\ee
  Therefore  for any  $ w \in\b {\rm S}(b)$ with  $b > 5 Q_\zeta $ we have \hp{\zeta}
\begin{equation}\label{110a1}
 \mG_{ii}^{(i,\emptyset)}
= - \frac 1 { w( 1+  m  )}  +  \cal E_1
\end{equation}
where
\begin{equation}\label{457}
 \cal E _1 =    w^{-1}\frac 1 { (1+  m)^2  }\Big [  m_ G^{( i ,i)}- m + Z^{(i)}_{i }    \Big ]  +  \OO\left ( \frac {|Z^{(i)}_{i }|^2 + \frac 1 { (N \eta)^2} } { |w| |1+  m|^3  } \right )  = \OO(\varphi^{ Q_\zeta/2} \Psi)
\end{equation}
where we have used \eqref{81} and $ |m_c| \sim |w|^{-1/2}$.
{ Together with} \eqref{80},  we thus have \hp{\zeta}
 \be\label{479}
\left | \im \mG_{ii}^{(i,\emptyset)} \right | \le  \left | \im \frac 1 { w( 1+  m  )} \right |
 +  \OO(\varphi^{ Q_\zeta/2} \Psi)  \le \im m_{\rm c} + C \Lambda +  \OO(\varphi^{ Q_\zeta/2} \Psi).
\ee

Using this estimate, \eqref{37mk} and  \eqref{110a1}, we can estimate $\cal Z_{i }:= \cal Z_{i }^{(\emptyset )}$ by
\be\label{485}
|\cal Z_{i }|
\leq  \varphi^{Q_\zeta/2} \sqrt{\frac{\im m_\mG^{(i,\emptyset )}+ |z|^2 \im\mG^{(i, \emptyset)}_{ii} }{N\eta}}
\le  \varphi^{Q_\zeta/2} \sqrt{\frac{\im m+ \im m_{\rm c} +  \Lambda +  \varphi^{Q_\zeta/2} \nc \Psi    }{N\eta}}  + \frac  {  \varphi^{Q_\zeta}}  { N \eta}
\le  \varphi^{   Q_\zeta } \Psi
\ee
 We can now use \eqref{485},   \eqref{110a1}  and  \eqref{37mk} to estimate the right hand side of   \eqref{110} such that
\begin{align}\non
 G _{ii}
 & =   - w^{-1}\left[1+  m_\mG^{(i, \emptyset )}+   |z|^2 \mG_{ii}^{(i, \emptyset)} +\cal Z_{i } \right]^{-1} \\
&   = - w^{-1}\left[1+  m- \frac { |z|^2 } { w( 1+ m)} + (m_\mG^{(i, \emptyset )} - m) +  \cal E_1   + \cal Z_{i } \right]^{-1}\label{110-29} \\
 & =   -   w^{-1} \left [ 1+ m  -\frac { |z|^2 } { w(1+ m)}      \right ]^{-1}  - \cal E_2   \label{110-31}
\end{align}
where  $\cal E_1$ and $\cal Z_{i }$ are bounded in  \eqref{457} and \eqref{485}  and  $\cal E_2$ is bounded by
$$
\cal E_2 =  \OO\left (  w^{-1} \left [ 1+ m  - \frac { |z|^2 } { w(1+ m)}      \right ]^{-2} \varphi^{  Q_\zeta}  \Psi  \right ) \le \OO(\varphi^{  Q_\zeta}  \Psi).
 $$
In the last inequality, we have used  \eqref{82} to bound $ 1+ m  - \frac { |z|^2 } { w(1+ m)}  $ and  \eqref{dA2} for $m_{\rm c}$.

Summing  over the index $i$ in \eqref{110-31},  we have
\begin{align}\label{110-315}
0
= w m +  \left [ 1+ m  -\frac { |z|^2 } { w(1+ m)}      \right ]^{-1}  + \OO( |w|\varphi^{   Q_\zeta} \Psi )
\end{align}
 Hence we have proved
$$
 0
  = w m   ( 1+ m)^2  -   m |z|^2      + 1+m  =  \OO \Big [    \big (|w| |m+1|^2+|z^2| \big )\varphi^{   Q_\zeta}  \Psi \Big ]
 $$
Together with the assumption \eqref{s3} on $\Lambda$ and \eqref{dA2} on the order of $m_{\rm c}$, this proves \eqref{110-31g}.
\end{proof}

 \begin{corollary}\label{xdz49}  Under the assumptions of Lemma \ref{cor:self}, the following properties hold.
Let $\bb T$, $\bb  U\in \llbracket 1, N\rrbracket$ such that $i\notin \bb T$ and $|\bb T|+| \bb N|\leq  C$.
 For any  $\zeta>0$
 and  $ w \in\b {\rm S}(b)$ for some   $b > 5 Q_\zeta $ with $Q_\zeta$ defined in Lemma \ref{lem:bh}, \nc
we have  \hp{\zeta}  for any  $i\in \bb U$ that
\be\label{2gsh}
G_{ii}^{(\bb T, \bb U)}-G^{( \emptyset, i)}_{ii}=\OO(\varphi^{  Q_\zeta}\Psi)\, .
\ee
If $i \not \in \bb U$, then
\be\label{3gsh}
G_{ii}^{(\bb T,\bb U)}-G _{ii}=\OO(\varphi^{   Q_\zeta}\Psi)  \, .
\ee
\end{corollary}

\begin{proof} We first prove the case $i \not \in \bb  U$.  We claim that the parallel version of \eqref{110-31}
holds as well, i.e.,
\be\label{77}
G_{ii}^{(\bb T, \bb U)}= -   w^{-1} \left [ 1+ m  -\frac { |z|^2 } { w(1+ m)}      \right ]^{-1}  + \OO(\varphi^{Q_\zeta}\Psi)
\ee
Comparing \eqref{77} with \eqref{110-31}, we have proved \eqref{3gsh}.

We now prove the case $i  \in \bb  U$.
By row-column symmetry, we have
$$
G^{( \bb T, \bb U)} =  \Big [ (Y^{(\bb T, \bb U)})^* Y^{(\bb T, \bb U)} - w \Big ]^{-1}  = \Big [ A^{( \bb U, \bb T)} (A^{( \bb U, \bb T)})^* - w \Big ]^{-1} := \mG (A)_{ii}^{(\bb U , \bb T)}\,   \quad A = Y^* .
$$
Hence we have   to prove,   for $i  \in \bb  U$ and $i \not \in \bb T$, that
$$
\mG (A)_{ii}^{(\bb U , \bb T)}-\mG (A)_{ii}^{(i  ,\emptyset)}=\OO(\varphi^{Q_\zeta}\Psi)  \, .
$$
We will omit  $A$ in the following argument. \nc

One can  extend  \eqref{110c1}-\eqref{457} to $\mG_{ii}^{(\bb U , \bb T)}$ and obtain
\be\label{78}
\mG_{ii}^{(\bb U , \bb T)}
= - \frac 1 { w( 1+  m  )}  +  \cal E_1^{(\bb T, \bb U)},\quad \cal E_1^{(\bb T, \bb U)}=\OO( \varphi^{Q_\zeta} \Psi)
\ee
as in \eqref{110a1}. Comparing \eqref{78} with the equation for $\mG_{ii}^{(i ,\emptyset)}$ \eqref{110a1}, we obtain \eqref{2gsh} in the case $i\in \bb U$.

\end{proof}

We define for any sequence $A_i$ ($1\leq i\leq N$) the quantity
$$
[A]:=  {\bf N^{-1} } \sum _i A_i.
$$
In application, we often use $ A=Z$ or $A=\cal Z$.   Define
$$
\cal D(m)=m^{-1} + w +  wm-\frac{|z|^2}{1+m}.
$$
The following lemma is our stability  estimate for  the equation   $ \cal D(m)=0$.
 Notice that it is a deterministic result. \nc It assumes that $|\cal D(m)| $ has a crude upper bound and then derives
a more precise  estimate on $\Lambda=|m-m_c|$.

\begin{lemma} [Stability  of the self-consistent  Equation] \label{lem:fm}   Suppose  that $1 - |z|^2  > t >  0$. Let
$\delta: \mathbb{C} \mapsto \mathbb{R}_+$ be a continuous function
  satisfying  the bound
\be\label{d2}
|\delta(w)| \leq (\log N)^{-8} |w^{1/2}|.
\ee
Suppose that,   for a  fixed $E$ with $ 0 \leq E \le C$ for some constant $C$ independent of $N$,    \eqref{s3}  and  the estimate
\be\label{d1}
|\Upsilon (m)(w, z)|  =|\cal D(m)m(1+m) (w, z)|\leq \delta(w) |m_{\rm c} |^2
 \ee
hold for $10\ge \eta\geq  \tilde \eta$ for some $\tilde \eta$ which may depend on    $N$.
 Denote $ \e^2 := \kappa + \eta $ where $\kappa= |E- \lambda_+|$ \eqref{37} in our case that $1 - |z|^2  > t >  0$.
 Then there is an $M_0$ large enough independent of $N$ such that  for any fixed $M>M_0$ and  $N$ large enough (depending on $M$)
the following estimates for $ \Lambda = |m-m_{\rm c}|$ hold for  $10\ge \eta\geq  \tilde \eta$:
\begin{align}
&  \text{Case 1}: \;    \Lambda \nc \le  \frac {M^{3/2} \delta} {  |w| }    \quad \text{or} \;
 \Lambda   \ge   \frac{1}{M ^2|w^{ 1/2}|}
\quad &
 \text{if}  \; \e^2  \ge 1/M^2  \label{c1} \\
& \text{Case 2a}: \;   \Lambda  \le  \frac {M \delta} { {\e } }     \quad \text{or} \;    \Lambda  \ge   \frac {2M \delta} { {\e } }   \quad
 & \text{if}   \;  \e^2   \le 1/M^2 \; \text{and}   \; \delta \le \frac { \e^2 } { M^{3/2}} \label{2a} \\
& \text{Case 2b}: \;  \Lambda  \le   M \sqrt \delta,    \quad \text{or} \;  \Lambda  \ge   2 M \sqrt \delta
 \quad &
\text{if} \;  \e^2  \le 1/M^2 \; \text{and} \; \delta \ge \frac { \e^2 } { M^{3/2}}  \label{2b}
\end{align}
The three upper bounds (i.e., the first inequalities in \eqref{c1}-\eqref{2b}) can be summarized as
 \be\label{ag24}
  \Lambda  \leq C  \frac{ \delta(w)|w|^{-1}}{ \sqrt{\kappa+\eta+\delta}} .
 \ee
 \end{lemma}

\begin{proof} Define the polynomial
$$
P_{w, z}(x) = w x (1+x)^2 +x(1-|z|^2) +1.
$$
By definition of $\Upsilon$ \eqref{110-31g},  we have
$$
 P_{w, z}(m)  = w m (1+m)^2 +m (1-|z|^2) +1 = \Upsilon = \cal D(m) m (1+m).
$$
Since $P_{w, z}(m_{\rm c})  = 0$, we have
$$
   w u^3+ B(w,z) u^2 + A(w,z) u =\Upsilon,   \quad u = m-m_{\rm c},
$$
$$
  B= w( 3 m_{\rm c}  + 2),
 $$
 $$
\quad A (w, z) = w (3 m_{\rm c} + 1) (m_{\rm c} + 1) + 1- |z|^2= 2 w m_{\rm c} (1+ m_{\rm c}) - \frac 1 m_{\rm c}.
$$
By definition of $P_{w, z}$, we can express $A$ and $B$ by
$$
P_{w, z}' (m_{\rm c} (w, z) ) =
A(w, z), \quad P_{w, z}'' (m_{\rm c} (w, z) ) =
2B(w, z).
$$

\noindent
{\it Case 1:}
In this case, we claim that the following estimates concerning  $A$ and $B$ hold:
\be\label{AB}
 |A| \ge C /M
, \quad   B =  \OO(|w^{ 1/2}|).
 \ee
Since $A$ and $B$ are explicit functions of $m_{\rm c}$,  equation  \eqref{AB} is  just properties  of the solution $m_{\rm c}$  of the third order polynomial
$P_{w, z}(m)$. We now give a sketch of the proof.
Consider first the case  $|w| \ll 1$.   Then \eqref{AB} follows from \eqref{dA2}, \eqref{26ssa}, \eqref{A20a} and the definitions of $A$ and $B$.

 We now assume  that $ w\sim 1$ .  Clearly,  $ |B|\leq  \OO(1) \sim |w^{ 1/2}|$, which gives \eqref{AB} for $B$.
To  prove $|A|\geq C/M$,  by  definition of $m_{\rm c}$  \eqref{defmc1}, we have $w= \frac{-1 - m_{\rm c} + m_{\rm c} |z|^2}{ m_{\rm c} (1 + m_{\rm c})^2 }$.
Thus we can rewrite
$A$ as
$$
A=\frac{-1 - 3 m_{\rm c} + 2 m_{\rm c}^2 (-1 + |z^2|)}{m_{\rm c}(1+m_{\rm c})}=\frac{2(-1 + |z^2|)}{m_{\rm c}(1+m_{\rm c})}(m_{\rm c}-a_+)(m_{\rm c}-a_-),
$$
$$
\quad a_\pm: =\frac{3 \pm \sqrt{ 1 + 8 |z|^2}}{ 4 (-1 + |z|^2)}=\frac{-2}{3\mp \sqrt{ 1 + 8 |z|^2}}\, .
$$
{ By  \eqref{dA2} and \eqref{ny27} (where $\al= \sqrt{ 1 + 8 |z|^2}$), we obtain \eqref{AB}.}

  We now prove \eqref{c1} by contradiction.
If \eqref{c1} is violated then with $u = m-m_c$ we have
$$
    |\Upsilon| = |u| | A(w,z)  +  B(w,z) u  + w u^2 | \ge     \frac { M^{3/2} \delta} {  |w| } \left [ {\frac C M} -  \frac {C_2}{ M^{2}  }
-  \frac { C_3} {M^4} \right ]   \ge   \frac { C \sqrt  M \ \delta} {   |w| },
$$
where $M$ is a large constant  in the last inequality.  By \eqref{d1} and \eqref{dA2},  $|\Upsilon| \le  C \delta  / |w|$. Thus
we have
$$
  \frac { C \sqrt  M \delta} {   |w| } \le  |\Upsilon|  \le    \frac {C \delta}  { |w|}
$$
 which is a contradiction provided that $M$ is large enough.

{\it Case 2:  $  \e^2 := \kappa + \eta  \le 1/M^2 $.} Note in this case $w\sim 1$. Then by \eqref{A18} we have
\be\label{c2}
 B  \sim 1,    \quad A(\lambda_+, z)  = 0
\ee
  where  the last equation can be checked  by direct computation   and we used $|z|^2<1-t<1$. There is a more intrinsic reason why the last equation for $A$ holds.
Notice that $\lambda_+$ is a point that the polynomial $P_{w, z} (m)|_{w = \lambda_+}$ has a double root. Therefore, we have  $0=P'_{w, z} (m_{\rm c} (\lambda_+, z) ) =
A(\lambda_+, z)$.

Notice that in the case  $\kappa + \eta  $  is   small enough, we can approximate $ A(w, z) $ by linearizing w.r.t. $w= \lambda_+$.
Thus by the defining equation $P'_{w, z} (m_{\rm c} (\lambda_+, z) ) =
A(\lambda_+, z)$, we have
 \begin{align}\label{B}
 A(w, z)  \sim   P_{w, z}'' (m_{\rm c} (\lambda_+, z) )
 (m_{\rm c} (w, z) - m_{\rm c} (\lambda_+, z))
+ \frac { \partial P_{w, z}} { \partial w}  (m_{\rm c} (\lambda_+, z) )  (w - \lambda_+)    \sim \sqrt {\kappa + \eta}= \e
\end{align}
where we have used that $P_{w, z}'' (m_{\rm c} (\lambda_+, z) ) = B(\lambda_+, z) \sim 1$,
$ \frac { \partial P_{w, z}} { \partial w}  (m_{\rm c} (\lambda_+, z) ) \sim 1$ and, by
 \eqref{A18},  that $ (m_{\rm c} (w, z) - m_{\rm c} (\lambda_+, z)) \sim \sqrt {\kappa + \eta}$.
While we can also check the conclusion of \eqref{B} by direction computation,  the current derivation provides a more intrinsic reason
why it is correct.

{\it Case 2a:}  Suppose \eqref{2a} is violated. We first
choose $M$ large enough so that $|m_{\rm c} (1 + m_{\rm c}) | \le  M^{1/4} $ in this regime.
Then by \eqref{c2} and \eqref{B}, with $w\sim 1 $,    we have  $$
  C  \delta  M^ {1/4}  \ge  |\Upsilon| = |u| | A(w,z)  +  B(w,z) u  + w u^2 | \ge  \frac{\delta M}{\e}
   \left [ C_1\e  -\frac{C_2 M\delta}{\e}
-  \frac { C_3  M^2  \delta^2} {\e^2  }  \right ] \ge C_1\delta  M /2,
$$
which is a contradiction  provided that $M$ is large enough. Here we have used that, by the restriction of $\e$ and $\delta$ in \eqref{2a} that $\e\ge  M^{3/4} \sqrt \delta$, $M$ is large enough constant   and $\delta\ll 1$.

{\it Case 2b:}  Suppose \eqref{2b} is violated.
 Similarly we have
 \begin{align*}
  C  \delta   M^{1/4}  \ge  |\Upsilon| &  =|u| |   B(w,z) u  +A(w,z)   + w u^2 |    \ge  |u|  \left [ C_1 M \sqrt  \delta   -C_2  \e -  C_3 M^2 \delta   \right ] \\ & \ge  C_1 |u|  \left [ M \sqrt  \delta/2   -C_2  \e   \right ]   \ge C_1 M^2  \delta/4
\end{align*}
which is a contradiction. Here we have used, by the restriction of $\e$ and $\delta$ in \eqref{2b} and $M$ is large enough constant, that
 $  C_2\e\le C_2 M^{3/4} \sqrt \delta \le  M \sqrt  \delta/20$.
\end{proof}

With a slighter strong condition on $\delta$ and an initial estimate $\Lambda\ll 1$ when
$\eta\sim 1$,  the first inequalities in \eqref{c1}-\eqref{2b}, i.e., \eqref{ag24},  always hold.
We state this as the following Corollary, which is a deterministic statement.

\begin{corollary}  [Deterministic continuity argument]\label{so51}
Suppose that  the assumptions of Lemma \ref{lem:fm} hold.  If  we have
$$
\Lambda(E+10\ii)\ll 1
$$
 and  that $\delta$ is decreasing in $\eta$ for $\e=\sqrt{\kappa+\eta}$  small enough,
 then \eqref{ag24} holds   all  $\eta \in   [\tilde \eta, 10]$.
\end{corollary}

\begin{proof}
By  assumption  $\Lambda(E+10 \ii)\ll1$ and
the left inequality of \eqref{c1} holds for $\eta = 10$.   By continuity of $\Lambda$, the same inequality,
$$ \Lambda  \le  \frac {M^{3/2} \delta} {  |w| },
$$
holds for $w=E+\ii\eta $ as long as
$\eta\in [\wt\eta, 10]$  and  $\e\geq 1/M$.

Suppose that as  $\eta$ decreases,
we get to Case 2a. Notice that when we decrease $\eta$,  by the conditions on $\e$
we will not go back to Case 1 from either Case 2a or Case 2b.
 For any $\e \le 1/M$ with $M$ large, we have
$$
  \frac {M^{3/2} \delta} {  |w| }   \le  \frac {M \delta} { {2\,\e } }.
$$
Hence at the transition point from Case 1  to  Case 2a,  the   inequality $  \Lambda(E+ i \eta)   \le \frac {M \delta} { {\e } } $  holds.
Thus by   continuity  of $\Lambda$,  the bound $ \Lambda(E+ i \eta)   \le  \frac {M \delta} { {\e } } $ in   \eqref{2b} holds until we leave Case   2a.

It is possible that we cross from Case 2a to Case 2b.
At the transition point, we have $ \delta = \frac { \e^2 } {M^{3/2}}$ and thus
$$
 \frac {M \delta} { {\e } }  \le \frac 1 2 M \sqrt \delta
$$
for $M$ large. Hence    the first inequality of Case 2b, i.e., $\Lambda \le M \sqrt \delta$    holds. By continuity, this bound continues to hold
 unless we leave Case 2b. Since $\delta$ is decreasing in $\eta$ when $\e$ is small,   once we get to Case 2b,  we will not go back to Case 2a (or Case 1 as explained before).

It is possible that the Case 2a is omitted and we get to Case 2b directly from Case 1. Notice that $\e=1/M$
at such a transition point and  we have  $|w| \sim 1$.
Furthermore,  by \eqref{d2},  we get $\delta \le 1/\log N  $  at the transition point.     Putting these together,  we have
for $M$ large,
$$
  \frac {M^{3/2} \delta} {  |w| }    \le \frac 1 2 M \sqrt \delta.
$$
Hence  the bound  $ \Lambda(E+ i \eta)   \le   M \sqrt \delta  $ in   \eqref{2b} holds.
\end{proof}

\subsection{The large $\eta$ case. }
Our method to estimate the Green functions and the Stieltjes transform is  to fix the energy $E$ and apply a continuity argument in $\eta$ by first showing that the crude bound in Lemma \ref{lem:fm}  holds for large $\eta$.
In order to start this scheme, we need to establish  estimates on the Green functions when $\eta=\OO(1)$. This is the main focus of this subsection.
We start with the following lemma which provide a  crude  bound on the Green functions.

\begin{lemma}
\label{lem:etaO1bd}
For    any $w\in {\rm S}(0)$ and $\eta >c>0$ for fixed $c$, we have the bound
\be\label{k430}
 \max_{i,j\notin U} |G^{(\mathbb{U},\mathbb{T})}_{ij}(w)| \leq C \;.
\ee
for some $C>0$. {  Notice that this bound is deterministic and is independent of the randomness. }
\end{lemma}
\begin{proof} By definition, we have
$$
\left|G_{ij}\right|=\left|\sum_{\al}\frac{\bu_\al(i)\overline\bu_ \al(j)}{\lambda_\al-w}\right|\leq \frac{1}{\eta}\sum_{\al} \bu_\al(i)\overline\bu_ \al(j) \leq\frac{1}{\eta}\leq C
$$
where we have used $\abs{\lambda_\al-w}\geq \im w=\eta $. Furthermore, $G^{(\mathbb{U},\mathbb{T})}_{ij}$ can be bounded similarly.
 \end{proof}

The main result of this subsection is the following bound on $\Lambda$.

\begin{lemma} \label{lem:selfcon1}   For any $\zeta>0$  and $\e>0$, we have
\be\label{gpww4}
\max_{w \in{ \b {\rm S}}(0), \eta= 10 }\Lambda(w)\leq  N^{-1/2+\e}
\ee
\hp{\zeta}.

\end{lemma}

\begin{proof}
  From \eqref{110c1}-\eqref{44},  for  $\eta = \OO(1)$  we have
$$
\left[\mG^{(i, \emptyset)} _{ii} \right]^{-1}
  =   - w\left[1+  m_ G^{( i ,i)}+Z^{(i)}_{i }   \right], \quad
|m_ G^{( i ,i)}- m| \le \frac { C} { N }\, .
$$
From \eqref{k430}, we have $|G_{ij}|+|\mG _{ij}|\leq \eta^{-1}\le \OO(1)  $   and
$|m_G^{(i,i)}|\le \OO(1) $. Hence the large deviation estimate \eqref{44} becomes,
\hp{\zeta},
\be\label{444}
| Z^{(i)}_{i }| \le  \varphi^{C_\zeta} \sqrt{\frac{\im m_G^{(i,i)}}{N}}
\le \varphi^{C_\zeta} N^{-1/2}.
 \ee
Thus  for any $\e>0$ we have
$$
\mG_{ii}^{(i,\emptyset)}:=- \frac 1 { w(1+  m  + \OO( N^{-1/2 + \e}) ) }
 $$
Together with  \eqref{110}, we obtain
$$
 G^{-1}_{ii}
  =   - w- wm_\mG^{(i, \emptyset)}+\frac{|z|^2}{1+  m  + \OO( N^{-1/2 + \e}) }-w   \mathcal Z_{i }.
$$
By an  argument similar to the one used in \eqref{444}, we can estimate $ \mathcal Z_{i }$
by
 $$
|\mathcal Z_i|\leq N^{-1/2+\e}
$$
for any  $\e>0$ \hp{\zeta}.  This implies that, \hp{\zeta},
 \be\label{110t}
 G^{-1}_{ii}
  =   - w- wm   +\frac{|z|^2}{1+m  +\OO (N^{-1/2+\e})}+  \OO (wN^{-1/2+\e}).
\ee
For any $\eta$ fixed, we claim that the following inequality between the real and imaginary parts of $m$ holds:
\be\label{xy}
|\re m| \le {2}\sqrt {\frac {\im m} \eta}.
\ee
To prove this,  we note that for any $\ell \ge 1$
\begin{align*}
N^{-1} \sum_{ |\lambda_j-E| \ge \ell \eta } \frac { E - \lambda_j} { (E- \lambda_j)^2 + \eta^2 }  & \le \frac 1 { \ell  \eta},\\
N^{-1} \sum_{ |\lambda_j-E| \le \ell \eta } \frac {| E - \lambda_j |} { (E- \lambda_j)^2 + \eta^2 }  & \le N^{-1} \sum_{ |\lambda_j-E| \le \ell \eta } \frac { \ell \eta } { (E- \lambda_j)^2 + \eta^2 }
\le  \ell \im m.
\end{align*}
Summing up these two inequalities and optimizing $\ell$, we have proved \eqref{xy}.

Assume that  $\im m \le c (\log N)^{-1}$. From \eqref{xy}, we have  $|m| \le  c (\log N)^{-1/2}$.
Together with $\im w = \eta \sim 1$,
 $$
|m| = N^{-1} \left | \sum_i  G _{ii} \right |
 = N^{-1} \left | \sum_i \left(- w- wm   + \frac{|z|^2}{1+m  } \right)^{-1} \right | +\OO (N^{-1/2+\e})\ge {  (-w+|z|^2+o(1))^{-1}}  \ge C
$$
for some constant $C$. This contradicts $|m| \le  c (\log N)^{-1/2}$ and we can thus assume that
$\im m \ge c (\log N)^{-1}$
when  $\eta \sim 1$ and $w=\OO(1)$. In this case, we also have
$$
|1+m|\geq C  (\log N)^{-1}.
$$
Then \eqref{110t} implies { for any $\e>0$} that \hp{\zeta}
$$
 G _{ii}
  =   \left(- w- wm   + \frac{|z|^2}{1+m  } \right)^{-1}+\OO (N^{-1/2+\e})
$$
Summing  up all $i$, we have the following equation for $m$ \hp{\zeta}:
$$
m
  =    \frac{-1-m  }{w(1+m)^2-|z|^2}  +\OO (N^{-1/2+\e})\, .
$$
We can rewrite this equation into the following form:
\be\label{4jj}
 P_{w, z} (m)=   w(1+m)^2m -|z^2|m+m+1=\OO (N^{-1/2+\e})\, .
\ee

It can be checked  (with computer calculation or rather complicated but elementary  algebraic calculation) that for $0 \leq E \leq 5\lambda_+$ and $\eta = O(1)$, the third order polynomial
$ P_{w, z} (m)$ has no double root and  there is only one root  with positive real part. We denote this root by  $m_1$ and the other two roots by  $m_2$ and $m_3$.  For $0 \leq E \leq 5\lambda_+$ and $t \le \eta \le t^{-1}$ for any $t$ fixed,  the three roots are separate by
order one due to compactness. Since there is no double root, we have $ |P'_{w, z} (m_1) | \ge c > 0$ whenever  $0 \leq E \leq 5\lambda_+$ and $t \le \eta \le t^{-1}$. Thus the stability of \eqref{4jj} is trivial and we have proved that in this range of parameters
$$
\left| m(w, z)- m_1(w, z)\right|=\OO (N^{-1/2+\e})
$$
for any $\e>0$  \hp{\zeta}.
\end{proof}

\subsection {Proof of the weak  local Green function estimates. }
  In this subsection, we finish the proof of  Theorem \ref{wempl}. We  fix an energy $E$  and we will decrease
the imaginary part $\eta$ of $w= E + i \eta$.   Recall all stability results are based on assumption \eqref{s3}, i.e.,   $ \Lambda \le \alpha   |m_c| \sim \alpha |w|^{-1/2}$ for some small constant $\alpha$,    which so far was established only for large $\eta$ in \eqref{gpww4}.  We would like to know that this condition continue to hold for smaller $\eta$.  More precisely,
 suppose that  \eqref{s3} holds  in a set $A$ for all $w=E+\eta i$ with   $\eta\in [\wt\eta, 10]$ where $\wt\eta$ satisfies
\be\label{eta1}
 \wt\eta \geq  \varphi^b N^{-1}|w|^{1/2}, \quad b >  5 Q_\zeta.
 \ee
 We can choose $ \tilde \eta =  \eta_1 < \eta_2 \ldots < \eta_n = 10$ such that $ |\eta_{i+1} - \eta_i | \le  N^{-20}$ and $n = O(N^{20})$.
By  \eqref{110-31g} and \eqref{gpww4}   we have   \hp{\zeta} in $A$,
\be\label{g1}
  \Upsilon (w)  \le  \OO(\varphi^{Q_\zeta}\Psi)(w)  \le   \varphi^{Q_\zeta} \sqrt {  \frac {|w|^{-1/2}} { N \eta}  }
\ee
for all $w = E + i \eta_j$ for all $1 \le j \le n$.
Since $\Lambda ( E + i \eta) $ is continuous in $\eta$  at a  scale, say,  $N^{-10}$,
\eqref{g1} holds for all  $\eta\in [\wt\eta, 10]$   \hp{\zeta} in $A$.
 Hence for $\wt \eta$ satisfying \eqref{eta1} the estimate   \eqref{d1}  holds with
$$
\delta=C\varphi^{Q_\zeta}  |w|\left(\frac{|w|^{-1/2} }{N\eta}\right)^{1/2}
$$
With this choice, we can check that the assumption on
$\delta$,   \eqref{d2},  holds as well. Furthermore $\delta$ is decreasing in $\eta$ when  $\e=\sqrt{\kappa+\eta}$ is small enough.
 By Corollary \ref{so51},   \eqref{ag24} holds   all  $\eta\in [\wt\eta, 10]$.

For $|z| < 1-t$ for some $ t >0$, if $\kappa \ll 1$ then
$|w| \sim 1$ and \eqref{ag24} implies
$$
\Lambda   \leq C \sqrt { \delta(w)}  \le  \varphi^{Q_\zeta /2} \left(\frac{1 }{N\eta}\right)^{1/4}.
$$
If $\kappa \ge c > 0$ for some $c> 0$ then
\be \label{xmb}
\Lambda   \leq C  \delta(w) |w|^{-1}  \le C  \varphi^{Q_\zeta }\left(\frac{|w|^{-1/2} }{N\eta}\right)^{1/2} \le
C \varphi^{ Q_\zeta }
 \frac {1}{|  w^{1/2}|}
\left(
\frac{|  w|^{ 1/2}}{N\eta}
\right)^{1/4} .
\ee
Combining both cases,  for any  $w\in \b {\rm S}(b)$, $ b >  5 Q_\zeta$,   we have  \hp{\zeta} in $A$ that
\be\label{411}
 \Lambda
 \le  \varphi^{ Q_\zeta}
 \frac {1}{|  w^{1/2}|}
\left(
\frac{|  w|^{ 1/2}}{N\eta}
\right)^{1/4} \le   C  \varphi^{ -Q_\zeta/5}  |  w|^{ - 1/2}  \sim  C  \varphi^{ -Q_\zeta/5} |m_{\rm c}|.
\ee

Suppose that $\hat  \eta := \wt\eta - N^{-20} \in \b {\rm S} (b)$ for some $b >  5 Q_\zeta$.
Then for any  $  \eta  \in [\wt\eta - N^{-20}, \wt \eta]$, by \eqref{411} and the continuity of $\Lambda$, we have
$$
\Lambda(E + i  \eta) \le \Lambda(E + i \wt \eta)  + N^{-10} \le  C  \varphi^{ -Q_\zeta/5}  |  w|^{ - 1/2} + N^{-10} \le \alpha |m_c(E + i \hat \eta)| /2
$$
Thus  the condition \eqref{s3}  in Lemma \ref{cor:self} is satisfied   \hp{\zeta}  in $A$.  Since we can start this procedure with $\wt \eta = 10$
and there are only $N^C$ steps to get to $\wt \eta =  \varphi^{5 Q_\zeta} N^{-1}|w|^{1/2}$, we have proved that \eqref{411} holds for all
$w \in \b {\rm S} (b)$ with  $b >  5 Q_\zeta$. Notice that from now on  the assumption \eqref{s3} holds \hp{\zeta}.

We can now  prove the estimate \eqref{res:wempl} on the diagonal term.    Comparing  \eqref{110-315} with \eqref{77}($\bb T=\bb U=\emptyset$),  for any  $w\in \b {\rm S}(b)$, $ b >  5 Q_\zeta$,  we have \hp{\zeta}
 \be\label{xmb2}
  |G_{ii}-m|\leq  O(\varphi^{  Q_\zeta }\Psi)
\ee
   By definition of $\Psi$, \eqref{411} and $m_c\sim|w^{-1/2}| $,
 we have
$$
  \Psi =   \left(\sqrt{\frac{\im m_C+\Lambda}{N\eta}}+\frac{1}{N\eta}\right)
\le    \left(\sqrt{\frac{|w|^{-1/2} }{N\eta}}+\frac{1}{N\eta}\right).
 $$
Using  the restriction on $\eta$ so that $N \eta \ge  |w|^{1/2} \varphi^{ 5 Q_\zeta}$, we have
\be\label{110-330}
\Psi
\le C \sqrt{\frac{ |w|^{-1/2} }{N\eta}}
\le C |w|^{-1/2} \left( \frac{\sqrt w}{N\eta}\right)^{1/4}.
\ee
With \eqref{xmb} and \eqref{xmb2}, we have thus proved that
$$
 \max_i  \big | G_{ii}- m_C  \big | \le  \varphi^{Q_\zeta}|w^{-1/2}|\left( \frac{\sqrt w}{N\eta}\right)^{1/4}
$$
for any   $w\in \b {\rm S}(b)$, $ b >  5 Q_\zeta$.
Hence the estimate \eqref{res:wempl} on the diagonal element  $G_{ii}$ holds.

To conclude  Theorem  \ref{wempl}, it remains to prove the estimate on the off-diagonal elements. Recall the identity  \eqref{110b} for $G_{ij}$ and the equations   \eqref{1321} and \eqref{1328}. We can estimate the off-diagonal
 Green function by
 $$ \Big | G_{ij} \Big |
  =  \Big | w         G_{ii}  G^{(i,\emptyset)}_{jj}
  |z|^2 \mG^{(ij, \emptyset )}_{ij} \Big |  +   \OO\left( \varphi^{Q_\zeta}
  \sqrt{\frac{\im m_\mG^{(ij,\emptyset)}
  +|z|^2\im\mG^{(ij, \emptyset)}_{ii}+|z|^2 \im\mG^{(ij, \emptyset)}_{jj}}{N\eta}}  \right)   , \quad i\neq j,
$$
\be\label{110b1}
 \Big | G_{ij} \Big |
  =  \Big |
  |z|^2 \mG^{(ij, \emptyset )}_{ij} \Big |  +   \OO\left( \varphi^{Q_\zeta}
  \Psi \right) , \quad i\neq j.
\ee
 Here we have used   $|G_{ii}  G^{(i,\emptyset)}_{jj}|=O(|w|^{-1})$, which follows from \eqref{2gsh}, $\Lambda\ll m_c$ and $|m_c|\sim |w^{-1/2}|$

 Recall the  identity   \eqref{110d} that
$$
 {\mG_{ij}^{(ij, \emptyset)}}
  =   -w\mG_{ii}^{(ij, \emptyset)}\mG^{(ij, i)}_{jj}\left( \mathrm y_i^{(ij)} G^{(ij,ij)}   \mathrm y_j^{(ij)*}\right), \quad i\neq j.
$$
By \eqref{132}, we have
$$
\left|
\left( \mathrm y_i^{(ij)} G^{(ij,ij)} ,  \mathrm y_j^{(ij)*}\right)
 \right|\leq
  \varphi^{Q_\zeta}
  \sqrt{\frac{|\im m_ G^{(ij,ij)}| }{N\eta}}\, .
$$
where we have used  \eqref{1328} and  that, by definition, $\im G^{(ij,ij)}_{ii}= 0= \im G^{(ij,ij)}_{jj} $.
Therefore, we have \hp{\zeta},
\be\label{110d2}
 \Big | {\mG_{ij}^{(ij, \emptyset)}} \Big |
\le \varphi^{Q_\zeta}
  \sqrt{\frac{ \im m_C+\Lambda+(N\eta)^{-1} }{N\eta}}\leq \varphi^{Q_\zeta} \Psi, \quad i\neq j,
\ee
where we also used $|\mG_{ii}^{(ij, \emptyset)}\mG^{(ij, i)}_{jj} |\leq C|m_c|^2\leq C |w|^{-1}$. Together with \eqref{110b1} and \eqref{2gsh}, we have  proved that \hp{\zeta}
\be\label{110b2}
 \Big | G_{ij}  \Big |
   \leq \varphi^{Q_\zeta}\Psi   , \quad i\neq j\, .
\ee
With \eqref{110-330}, it   proves    Theorem  \ref{wempl} for the off-diagonal elements
 provided that  $w\in \b {\rm S}(b)$ with  $ b >  5 Q_\zeta$. Finally, we rename $b$ as the  $C_\zeta$ and this concludes
the proof of Theorem  \ref{wempl}.

\section{Proof of the strong local Green function estimates}\label{sec: PS}

Lemma \ref{cor:self}  provides an error estimate to  the self-consistent equation of $m$
 linearly  in $\Psi$.
 The following Lemma improves this estimate  to quadratic in $\Psi$.
This  is  the key improvement leading to a proof  of the
strong  local Green function estimates, i.e., Theorem \ref{sempl}.

\begin{lemma} \label{lem:Zlem}
 For any $\zeta>1$,  there exists $R_\zeta > 0 $ such that the following statement holds. Suppose
 for  some deterministic number $\wt \Lambda(w, z)$  (which can depend on $\zeta$)  we have
$$
  \Lambda(w, z) \leq \wt \Lambda(w, z)  \ll m_c (w, z)
$$
for  $ w \in{ \b {\rm S}}( b)$, $ b > 5 R_\zeta$,  in a set
$\Xi$
with  $\P(\Xi^c) \leq e^{-p_{N}(\log N)^2 }$ and $p_N$ satisfies that
\be\label{kk20}
\varphi { \le} p_N{ \le} \varphi^{  2\zeta}.
\ee
 Then there exists  a  set $\Xi'$   such that  $ \P(\Xi'^c) \leq e^{-p_{N} } $  and
  \be\label{D521}
\cal D (m(w,z))\leq  \frac{1}2 \varphi^{R_\zeta} |m_{\rm c}|^{-3} \wt \Psi  ^2 , \quad
   \widetilde \Psi
   :=\sqrt{
   \frac{\im \,m_{\rm c}+\wt \Lambda}{N\eta}
   }
   +\frac{1}{N\eta},\quad  {\rm in}\quad \Xi'.
\ee
Notice that the probability deteriorates in the exponent by a $(\log N)^{-2}$ factor.
 \end{lemma}

We remark  that,  by Lemma \ref{pmcc1},  $\im m_{\rm c} \ll |m_{\rm c}|$ when $\eta + \kappa \ll 1$.  Hence we have to track the dependence of
$\im m_{\rm c}$ carefully in the previous Lemma. This is one major difference between the weak and strong local Green function estimates.    Similar phenomena occur for the Stieltjes transforms of the  eigenvalue distributions of
 Wigner matrices.  Lemma \ref{lem:Zlem} will be proved later in this section; we now use it to prove Theorem \ref{sempl}.
We first give a {\it heuristic} argument.

Suppose that  we have the estimate \eqref{D521} with $\wt \Psi$ replaced by $\Psi$.
We assume $\Lambda\ge (N\eta)^{-1} $ for convenience so that $\Psi^2 \sim (\im m_{\rm c}+\Lambda)/(N\eta) $
(If this assumption is violated then   then \eqref{res:sempl} holds automatically and we have nothing to prove).
Then we can apply  Corollary \ref{so51} by choosing
\be\label{deltaxiao}
 \delta = \varphi^{R_\zeta}|w|^{3/2}  \left [ \frac{\im m_{\rm c}+\Lambda}{N\eta}    \right ]
\ee
 which implies \eqref{ag24}.  Consider first the case $\kappa + \eta \sim \OO(1)$.   Using \eqref{ag24} with the choice of $\delta$ in
\eqref{deltaxiao} and  $\kappa+\eta+\delta \ge \OO(1)$,  we have
$$
 \Lambda
 \le   \varphi^{R_\zeta} |w|^{1/2}  \left [ \frac{\im m_{\rm c}+\Lambda}{N\eta}   \right ].
$$
When $\eta$ satisfies the condition \eqref{eta1},
the coefficient of  $\Lambda$  on the right side of the last equation  is smaller than $1/2$.
 Hence, using
 $\im m_{\rm c}\leq |m_{\rm c}|\leq    C |w|^{-1/2}$ (see Proposition \ref{tnf}), we have
$$   \Lambda   \le C  \varphi^{R_\zeta }\left [ \frac{ |w|^{ 1/2}  \im m_{\rm c}}{N\eta}    \right ]  \le  C\varphi^{R_\zeta } \frac{1}{ N\eta }\, .
$$
 We now consider the case $\kappa + \eta \ll 1$ and thus $|w| \sim \OO(1)$.
 From  the first inequality of \eqref{ag24}, we have
\be\label{y1}
\Lambda    \leq C \frac{ \delta(w)|w|^{-1}}{ \sqrt{\kappa+\eta+\delta (w)}} \leq C  \sqrt {\delta(w) }.
 \ee
Also, in the regime $\kappa + \eta \ll 1$,   \eqref{esmallfake} asserts that
 $$
\im m_{\rm c} \le
C \sqrt{\kappa+\eta} ,  \quad \frac{  \im m_{\rm c} }{  N \eta \sqrt{\kappa+\eta+ \delta }}
\le  \frac{ C   }{  N \eta }\, .
$$
 Using  the choice of $\delta$ in
\eqref{deltaxiao},  we have
$$
  \Lambda     \leq C \varphi^{R_\zeta}  |w|^{1/2} \frac{ \im m_{\rm c}+\Lambda }{ N \eta  \sqrt{\kappa+\eta+\delta}}
 \leq  C\varphi^{R_\zeta} \frac{1 }{ N \eta  }  + C \varphi^{R_\zeta}  \frac{  \Lambda }{ N \eta  \sqrt{\kappa+\eta+\delta}}
\le C' \varphi^{R_\zeta} \frac{1 }{ N \eta  }
  $$
where we have used  \eqref{y1} to  absorb the last term involving $\Lambda$ in the last inequality with a change of constant $C$.
This completes the heuristic proof of
Theorem \ref{sempl}. We now give a formal proof of this theorem assuming Lemma \ref{lem:Zlem}.

\noindent

\begin{proof}[Proof of Theorem \ref{sempl}]
 We first prove  \eqref{Lambdaofinal} assuming
 \eqref{res:sempl}.  By
\eqref{110b2} and the definition of $\Psi$, we have for $i\neq j$,
 $$
 \Big | G_{ij}  \Big |
  \leq \varphi^{{ R_\zeta}} \left [  \sqrt{
   \frac{\im \,m_{\rm c}+ \Lambda}{N\eta}
   }  +\frac{1}{N\eta} \right ] \le \varphi^{{ R_\zeta}} \left [  \sqrt{
   \frac{\im \,m_{\rm c}}{N\eta}
   }  +\frac{1}{N\eta} \right ]
$$
where we have used  \eqref{res:sempl} in the last step. This proves \eqref{Lambdaofinal}.

The main task in  proving Theorem \ref{sempl}  is to prove \eqref{res:sempl}. {We first  consider the case   that $|z| \le 1-t$. }
We   assume that $ \zeta$ is large enough, e.g.,  $\zeta\geq10$.  By Theorem \ref{wempl} and $m_c \sim |w|^{-1/2}$
\eqref{dA2}
for $|z|  < 1-t$,        there exists a constant  $C_{\zeta+5}$
such that for any $ w \in \b {\rm S}(b)$, $b > 5 C_{\zeta+5}$ and {  $\alpha \ll 1$, we have
\be\label{xy33}
\Lambda( w)\leq \Lambda_1:= \alpha |m_{\rm c}| \sim    O(\alpha |w|^{-1/2}),
 \ee
holds   with the probability larger than
$1-\exp(-\varphi^{\zeta+5})$  (here we have replaced $\zeta$ in Theorem \ref{wempl} by $\zeta + 5$ for the convenience of the following argument).
Since $\b {\rm S}(b)$ is decreasing in $b$, we can choose  $D_\zeta = 5
\max (C_{\zeta+5}, R_\zeta) $   so that we can apply Lemma \ref{lem:Zlem}
with $p_N = \varphi^{\zeta+5}$ (which guarantees \eqref{kk20}). Together with $\Lambda_1 \le |m_c|$, we   have, for any  $  w \in \b {\rm S}(D_\zeta)$ fixed,
\be\label{c513}
\cal D (m)\leq \frac{1}2\varphi^{R_\zeta} |m_{\rm c}|^{-3} \Psi _1^2 , \quad
     \Psi_1
   :=\sqrt{
   \frac{\im \,m_{\rm c}+ |m_c| }{N\eta}}
   +\frac{1}{N\eta},
 \ee
holds   with the probability larger than
$1-\exp(-\varphi^{\zeta+5}(\log N)^{ -2})$. }
  Notice that the application of Lemma \ref{lem:Zlem} causes the probability  in the exponent to  deteriorate by a $(\log N)^{-2}$ factor.

Using  \eqref{c513}, we  can  apply  Corollary \ref{so51} with
\be\label{514ng}
\delta=\delta_1:= \varphi^{ R_\zeta} |m_{\rm c}|^{-3} \Psi _1^2.
\ee
 Here the assumption of $\Lambda(E+10\ii)$ is guaranteed by \eqref{xy33}. By definition of $\Psi_1$  \eqref{c513} and $|m_c| \sim |w|^{-1/2}$ \eqref{dA2}, for  $w \in \b {\rm S}(D_\zeta)$,   we have
$$
\delta\le \varphi^{ R_\zeta}\frac{|w|} {N\eta}  \ll (\log N)^{-8}  |w|^{1/2}.
$$
Furthermore, it is easy to prove that $\delta$ is decreasing in $\eta$ when   $\kappa+\eta $ is small.
We have thus  verified the assumptions on $\delta$  in Corollary \ref{so51}
with the choice  $\delta= \delta_1$ given in \eqref{514ng}.  From \eqref{ag24}, we  obtain for  $w \in \b {\rm S}(D_\zeta)$,  { with $C_0$ being the $C$ in \eqref{ag24}, }
$$
 \Lambda \leq C_0\frac{\delta_1 |w|^{-1}}{\sqrt{\kappa+\eta+\delta_1}}\leq  C_0  \frac{ \varphi^{  R_\zeta} }{N\eta \sqrt{\kappa+\eta+\delta_1}}
$$
holds   with the probability larger than
$1-\exp(-\varphi^{\zeta+5}(\log N)^{- 2})$.
We have thus proved  \eqref{res:sempl} provided  that $\kappa+\eta\geq (\log N)^{-1}$.

We now prove  \eqref{res:sempl} when  $\kappa+\eta\leq  (\log N)^{-1} $. We have in this case $ |w|\sim 1$. We
apply Lemma \ref{lem:Zlem} with $\wt \Lambda = \Lambda_1= |m_c| \sim 1 $ given by \eqref{xy33}. Thus \eqref{c513} holds
and we apply Corollary   \ref{so51}  with $\delta= \delta_1$ \eqref{514ng}.
Since   $\Lambda_1\ge (N\eta)^{-1}$ and $\im m_c\sim \sqrt{\kappa+\eta}$ \eqref{esmallfake}, the conclusion
of Corollary \ref{so51} implies that for  $w \in \b {\rm S}(D_\zeta)$,
$$
  \Lambda     \leq C_0 \varphi^{R_\zeta}  |w|^{1/2} \frac{ \im m_{\rm c}+\Lambda_1 }{ N \eta  \sqrt{\kappa+\eta+\delta_1}}
 \leq  C_1\varphi^{R_\zeta} \frac{1 }{ N \eta  }  + C_1 \varphi^{R_\zeta}  \frac{  \Lambda_1 }{ N \eta  \sqrt{\delta_1}}
$$
 holds   with probability larger than
$1-\exp(-\varphi^{\zeta+5}(\log N)^{- 2})$. {Here $C_1$ depends only  on $C_0$. }
From the definition of $\delta_1$ and $\Psi_1$, we have
$$
 \varphi^{R_\zeta} \frac{  \Lambda_1 }{ N \eta  \sqrt{\delta_1}} \le   \varphi^{ R_\zeta/2}
 \frac{ |m_c|^{3/2} }{ N \eta } \frac{  \Lambda_1 } { \Psi _1} \leq  C_2\varphi^{ R_\zeta /2 } \left(\frac{\Lambda_1 }{N\eta}\right)^{1/2},
$$
 where for the last inequality we used
 $$
 \Psi_1\ge \sqrt{ \Lambda_1 /( N\eta)}.
 $$
 Since $\Lambda_1\ge(N\eta)^{-1} $, combining the last two inequalities,   for  $w \in \b {\rm S}(D_\zeta)$,   we have
\begin{equation}\label{518hh}
 N \eta  |\Lambda|
  \leq C_3  \varphi^{  R_\zeta }+
 C_3     \varphi^{  R_\zeta/2 } \left(N\eta \Lambda_1 \right)^{1/2}
    \le   \varphi^{   R_\zeta } \left(  N\eta \Lambda_1 \right)^{1/2}
\end{equation}
 holds   with the probability larger than
$1-\exp(-\varphi^{\zeta+5}(\log N)^{-  2})$ { for some $C_3$}. Notice that we have used $ N \eta  \ge \varphi^{5 R_\zeta}$
in the last step in \eqref{518hh}.

Repeating this process   with the choices
$$
N \eta \Lambda_2 { :}= \varphi^{   R_\zeta } \left(N\eta \Lambda_1 \right) ^{1/2}
,\quad     \Psi_2
   :=\sqrt{
   \frac{\im \,m_{\rm c}+\Lambda_2}{N\eta}
   }
   +\frac{1}{N\eta}
   ,\quad \delta_2:= \varphi^{  R_\zeta} |m_{\rm c}|^{-3} \Psi _2^2,
$$
  for $w \in \b {\rm S}(D_\zeta)$,   we obtain that
 $$
 N \eta  |\Lambda| \leq     { C_3}  \varphi^{ R_\zeta }+
 {C_3}    \varphi^{ R_\zeta/2 } \left(N\eta \Lambda_2 \right)^{1/2}
  \le     \varphi^{   R_\zeta } \left( N\eta  \Lambda_2  \right)^{1/2}
 $$
holds   with the probability larger than
$1-   \exp(-\varphi^{\zeta+5}(\log N)^{-4})$.
{ Notice that the last constant $C_3$ is the same as the one appears in \eqref{518hh} and it does not change in the iteration
procedure. }
 We now
iterate this process $K$ times
to have
$$
N \eta  |\Lambda|
    \le    \varphi^{ R_\zeta } \left( N\eta  \Lambda_K  \right)^{1/2}
    \le   \varphi^{ 2 R_\zeta }\left( N\eta \Lambda_1 \right)^{1/2^K}
 $$
holds   with the probability larger than
$1-  \exp(-\varphi^{\zeta+5}(\log N)^{- 2 K})$.  We need $K$ so large that
$$
\left( \Lambda_1 N\eta \right)^{1/(2^K)} \le   (C N)^{1/(2^K)}
\leq \varphi,
$$
i.e.,
$$
K\geq \frac{\left(\log \log ( CN)-\log \log \varphi \right)}{\log 2} = \frac{\left(\log \log ( CN)-2\log \log \log  N  \right)}{\log 2}
$$
On the other hand, we need $K$ small enough so that
\be
1-\exp(-\varphi^{\zeta+5}(\log N)^{-2K})\geq1 -\exp(-\varphi^{\zeta}), \quad \text{ i.e.,} \;   \varphi^{5}(\log N)^{-2K} \geq  1 .
\ee
{We note that it also guarantees \eqref{kk20}, since $\varphi^{\zeta+5}\ge p_1\ge p_2\ge\cdots \ge p_K\ge \varphi$. }
We choose $K = \log \log N/\log 2 $ and
  we have thus proved that
\be\label{521}
N\eta  |\Lambda| \leq
 \varphi^{ 2 R_\zeta+{1}}  \ee
with the probability larger than
$1- \exp(-\varphi^{\zeta})$ which  implies  \eqref{res:sempl} when  $\kappa+\eta\leq (\log N)^{-1}$.
This  completes the proof of  Theorem \ref{sempl}. \nc
\end{proof}

\subsection{Proof of Lemma \ref{lem:Zlem}. }
The first step in proving Lemma \ref{lem:Zlem} is to derive a second order self-consistent equation which identifies
the first order dependence of the correction in the self-consistent equation derived in Lemma \ref{cor:self}.
The second error terms will be bounded by $\Psi^2$; the first order terms are of the forms of
averages of $Z^{(i)}_i$ and $\cal Z_i$.
In Lemma \ref{lem:32y},  the averages of $Z^{(i)}_i$ and $\cal Z_i$ will be estimated  by $\Psi^2$.  This  improvement from the naive
order $\Psi$ to $\Psi^2$ is the key ingredient to obtain
the strong local law. We remark  that $\im m_{\rm c} \ll |m_{\rm c}|$ when $\eta + \kappa \ll 1$. Hence
the dependence of  $\im m_{\rm c}$ verses  $m_c$  has to be tracked carefully. We now state  the second order self-consistent equation:
as the following lemma.

\begin{lemma} [second order self-consistent equation]\label{cor:Dm} For any constant $\zeta>0$, there exists $C_\zeta>0$
such that for $  w\in \b {\rm S}(b)$, $b\ge 5C_\zeta$ \hp{\zeta}
\begin{align}\label{459g}
 \cal D(m)   \leq \OO\left(\varphi^{C_\zeta}   \frac{1 }{m_{\rm c}^3}  \Psi^2 + w[\mathcal Z]+m_{\rm c}^{-2}[ Z_\ast^\ast]\right)
 \end{align}
where
$$
[ Z_\ast^\ast] = N^{-1} \sum_i Z^ {(i)}_{i}, \quad [ \mathcal Z] = N^{-1} \sum_i \mathcal Z_{i}\, .
$$
  \end{lemma}

\noindent
\begin{proof}
 We  have proved  the weak local Green function estimate, i.e., Theorem \ref{wempl}, in Section \ref{sec:PTs}.
This in particular implies  that \eqref{s3} holds \hp{\zeta} in $\b {\rm S}(b)$ for large enough $b$  \hp{\zeta}.  With this remark in mind, we now prove
Lemma \ref{cor:Dm}.

We first take the inverse of both sides of  \eqref{110-29}  and   sum  up $i$ to get,  \hp{\zeta},
\begin{align}\label{449t1}
 {  N^{-1} } & \sum_i G_{ii}^{-1}=-w-wm+\frac{|z|^2}{1+m}+w[\mathcal Z]  { - \frac {|z|^2}  { (1+ m)^2}  [ Z_\ast^\ast] } \\\nonumber & +  {  N^{-1} } \sum_i \OO\left ( \frac { (Z^ {(i)}_{i})^2 + \frac 1 { (N \eta)^2}   }{ (1+ m)^3} \right )
+ |w| \OO( \frac 1 N\sum_i m_\mG^{(i,\emptyset)}-m)+|m_{\rm c}|^{-2}\OO\left ( \left | \frac1N\sum_i m^{(i,i)}-m \right |  \right ),
\end{align}
where we have used \eqref{457} and the bound \eqref{81}.
Recall the  estimates of $\mathcal Z_i$ and $Z^{(i)}_i$  by $\Psi$  in \eqref{44} and \eqref{485}.
Hence we have
\begin{align}\label{449t}
 {  N^{-1} }  \sum_i G_{ii}^{-1}&=-w-wm+\frac{|z|^2}{1+m}+\varphi^{C_\zeta}  \OO(m_{\rm c}^{-3} \Psi^2) \\\nonumber
  & +\OO(w[\mathcal Z ])+\OO(m_{\rm c}^{-2}[ Z_\ast^\ast] )
+ |w| \OO( \frac 1 N\sum_i m_\mG^{(i,\emptyset)}-m)+|m_{\rm c}|^{-2} \OO\left ( \left | \frac1N\sum_i m^{(i,i)}-m \right |  \right ).
\end{align}
{ By \eqref{xmb2}-\eqref{110-330}, we have
\be\label{dnn2}|G_{ii}-m|\le\OO(\varphi^{Q_\zeta}\Psi) \ll |m_{\rm c}|,
\ee
where $b\geq 5 Q_\zeta$ and $Q_\zeta$ is defined in Lemma \ref{lem:bh}.  } We now  perform the expansion  $G_{ii} ^{-1} = [(G_{ii}-m) + m] ^{-1}$ to have
$$
 G_{ii} ^{-1}=m ^{-1}- \frac{G_{ii}-m}{m^2}+O( \varphi^{{ 2}Q_\zeta} |m_{\rm c}|^{-3}\Psi^2).
$$
Using this approximation in
\eqref{449t}, we have
\begin{align}\label{449w}
m^{-1} +w+wm-\frac{|z|^2}{1+m}  = & \varphi^{{ 2}Q_\zeta}  \OO(m_{\rm c}^{-3} \Psi^2)  +\OO(w[\mathcal Z ])+\OO(m_{\rm c}^{-2}[ Z])  \\\label{449w4} &
+ |w| \OO( \frac 1 N\sum_i m_\mG^{(i,\emptyset)}-m)+|m_{\rm c}|^{-2} \OO\left ( \left | \frac1N\sum_i m^{(i,i)}-m \right |  \right ).
\end{align}
Using  \eqref{35bd}, we have
$$
\frac 1 N\sum_i m_\mG^{(i,\emptyset)}-m=\frac 1 N\sum_i m_G^{(i,\emptyset)}-m +\frac{C}{Nw}.
$$
Furthermore,  with  (\ref{111}) we have
\be\label{531tt}
m_ G^{(i,\emptyset)}-m
=
\frac1N
\left( G_{ii}+\sum_{j \neq i}\frac{ G_{ji} G_{ij} }{ G_{ii}}\right)=\frac1N\sum_j \frac{ G_{ji} G_{ij}}{ G_{ii}} =   \OO(\frac{\im  G_{ii}}{N\eta |G_{ii}|}).
\ee
The diagonal element  { $G_{ii}$ can be estimated by \eqref{dnn2}  so that}
$$
\left | \frac{\im G_{ii}}{N\eta |G_{ii}|} \right | \le \varphi^{Q_\zeta}   \frac{\im m_{\rm c}+\Lambda + \Psi   }{N\eta |m_{\rm c}|}
\le \varphi^{Q_\zeta}  \frac {\Psi^2 } { |m_{\rm c}| }.
$$
Therefore, we have
\be\label{452}
\OO( \frac1N\sum_i m_\mG^{(i,\emptyset)}-m)\leq \OO( \frac1N\sum_i m_G^{(i,\emptyset)}-m)  +\frac{C}{N  |w| }
\le \varphi^{Q_\zeta}  |m_{\rm c}|^{-1} \Psi^2 +\frac{C}{N  |w| }.
\ee
Notice that only the imaginary part of $m_{\rm c}$ appears through $\Psi$ instead of $m_{\rm c}$ which can be much bigger
near the spectral edge.

We now estimate the last term in \eqref{449w4}. Notice that  $\mG^{(i, \emptyset)}$ is the Green function of the matrix $A^+ A$ where $A = (Y^{(i, \emptyset)})^*$.
Then $m^{(i,i)}$ is the Green function of $A^{(i, \empty), +} A^{(i, \empty)}$ where we have used
$A^{(i, \empty)} = Y^{(i, i)}$. Thus we can apply \eqref{531tt}
 (which holds for matrices of the form $A^+ A$ with A not necessarily a square matrix)   to get
$$
|m_\mG^{(i,\emptyset)} -  m^{(i,i)}|\leq  \OO(\frac{\im \mG^{(i,\emptyset)}_{ii}}{N\eta |\mG^{(i,\emptyset)}_{ii}|}) . $$
By \eqref{479}, we have
$$
 \im \mG^{(i , \emptyset)}  _{ii}\leq C \left(\im  m_{\rm c}+  \Lambda +  \varphi^{C_\zeta} \Psi \right).
$$
By \eqref{457} and \eqref{110a1},
$$
|\mG^{(i , \emptyset)}_{ii}|\sim |w^{-1/2}|\sim |m_{\rm c}|  \, .
$$
These  estimates imply  that
 \be\label{456}
 \left | \frac1N\sum_i m^{(i,i)}-m \right | \le \left |  \frac1N\sum_i m_\mG^{(i,\emptyset)}-m \right | +  \frac1N\sum_i|m^{(i,i)}-m_\mG^{(i,\emptyset)}|\leq \varphi^{ Q_\zeta} |m_{\rm c}|^{-1}\Psi^2.
\ee
Inserting \eqref{452} and \eqref{456} into \eqref{449w}, we obtain
$$
 \cal D(m)   \leq \OO\left(\varphi^{{ 2}Q_\zeta}\left(  \frac{1 }{m_{\rm c}^3}  \Psi^2+N^{-1}\right)+ w[\mathcal Z]+m_{\rm c}^{-2}[ Z_\ast^\ast]\right).
$$
To conclude Lemma \ref{cor:Dm}, we choose { $C_\zeta=2Q_\zeta$} and it remains to  prove $
|\frac{1}{m_{\rm c}^3}\Psi^2|\geq \OO(N^{-1})$.
By  definition of $\Psi$ and the fact that $ |m_{\rm c}| \sim |w|^{-1/2}$  \eqref{dA2},   this inequality follows from  the following
property of $\im m_c$:
$$
| \frac{\im m_{\rm c}}{N\eta}|\geq \OO(N^{-1}).
$$
This estimate  on $\im m_c$ is a direct consequence of    \eqref{A17}, \eqref{esmallfake},  \eqref{A20a} and \eqref{A21}.
This completes the proof of Lemma \ref{cor:Dm} ({ with  $C_\zeta$ increasing by 1}).

 \end{proof}

We now
estimate the averages $[\mathcal Z]$ and $  [ Z_\ast^\ast] $. Our goal is to catch cancellation effects due to the average over the indices $i$.
This  is the content of the next lemma, to be proved  in next subsection.  Clearly this lemma  completes the proof of Lemma
\ref{lem:Zlem}.

 \begin{lemma}\label{lem:32y} For any $\zeta>1$,  there exists $R_\zeta > 0 $ such that the following statement holds. Suppose
 for  some deterministic number $\wt \Lambda(w, z)$  (which can depend on $\zeta$)  we have
$$
  \Lambda(w, z) \leq \wt \Lambda(w, z)  \ll m_c (w, z)
$$
for  $ w \in{ \b {\rm S}}( b)$, $ b > 5 R_\zeta$,  in a set
$\Xi$  with  $\P(\Xi^c) \leq e^{-p_{N}(\log N)^2 }$ and $p_N$ satisfies that
\be\label{kk20n}
\varphi { \le} p_N{ \le} \varphi^{  2\zeta}.
\ee
 Then there exists  a  set $\Xi'$   such that  $ \P(\Xi'^c) \leq e^{-p_{N} } $  and
\be \label{32youn}
\big| [\cal Z]\big|+\big |[ Z_\ast^\ast]\big|
\leq \varphi^{C_\zeta} |w|^{1/2} \widetilde \Psi ^2, { \quad in\quad  \Xi' }
  \ee
 where $\wt \Psi$ is defined in \eqref{D521}.
   \end{lemma}

 \subsection{Strong bounds on $[Z]$.  }
 In this subsection, we prove  Lemma \ref{lem:32y}. The main tool is the abstract cancellation Lemma \ref{abstractZlemma}.

We first  perform a cutoff for all random variables $X_{ij}$ in $X$
so that  $ |X_{ij}| \le N^{10}$. Due to the subexponential decay assumption, the probability of the complement of this event  is $e^{-N^c}$,
which is negligible.

Define $P_i$ and $\cal P_i$ as the operator for the  expectation value  w.r.t. the $i$-th row and $i$-th column. Let
$$
Q_i=1-P_i,\quad \cal Q_i=1-\cal P_i$$
{ With this convention and Lemma \ref{idm}, we can rewrite $ \cal Z_i$ and $Z_i^{(i)}$, from Definition \ref{Zi-def}, } as
$$
  \cal Z_i=\cal Q_i \left(wG_{ii}\right)^{-1}, \quad Z_i^{(i)}=  Q_i \left(w\mG^{(i, \emptyset)}_{ii}\right)^{-1}.
$$
 By definition, for any $i,j, \bb U, \T$, we know  $|G^{\bb U, \T}_{ij}|\leq \eta ^{-1}$.
From the identities of $G_{ii}$ and $\mG^{(i, \emptyset)}_{ii}$ in Lemma \ref{idm}
 and $|X_{ij}|\leq N^C$,  we have,   for any $1\le i\le N$,
\be\label{bdl}
 |G_{ii}|^{-1}
+|\mG^{(i, \emptyset)}_{ii}|^{-1}
\leq N^C.
\ee

   Let  $D_\zeta=\max\{C_{6\zeta+ 10}, Q_{6\zeta+ 10}+1\}$ with $C_{\zeta}$  defined  in Lemma \ref{wempl} and $Q_\zeta$ in Lemma \ref{lem:bh}.
Then for any  fixed $\bb T, \bb U$: $|\bb T|$, $|\bb U|\leq p$  there exists a set { $\Xi_{  \bb T, \bb U}$}
with
$$
P(\Xi_{  \bb T, \bb U})\geq 1-e^{-\varphi^{{ 6} \zeta+10}}
$$
 such that for  any ${w}\in \b {\rm S}(b)$, $b>5D_\zeta$   the following properties hold.

\begin{enumerate}
 \item for $w \in{ \b {\rm S}}(b )$
\be\label{wwxx}
 \Lambda\le \varphi^{-D_\zeta/4}|w^{-1/2}|, \quad \Psi\leq \varphi^{-2D_\zeta}|w^{-1/2}|
 \ee
\item for $w \in{ \b {\rm S}}(b )$
  \be\label{res:wemplnew}
  \max_{ij}|G_{ij}(z)-m_{\rm c}(z)\delta_{ij}|
  \leq \varphi^{D_\zeta}\frac {1}{|  w^{1/2}|}
\left(
\frac{|  w^{1/2}|}{N\eta}
\right)^{1/4} , \quad b > 5 D_\zeta.   \ee

\item   for any $i\neq j$,
 \be \label{559c1}
|(1-\E_{\by_i})\by_i^*  \cal G^{(i \bb T,  \emptyset)}  \by_i|
 +|\by_i^*  \cal G^{(ij \bb T,  \emptyset)}  \by_j|\leq \varphi^{D_\zeta}\Psi
 \ee
 \be\label{559c2}
 |(1-\E_{ \mathrm y_i})\mathrm y_i^{(i)} G^{(i ,  i \bb U)} ( \mathrm y_i^{(i)}) ^* |
 +| \mathrm y_i^{(i)} G^{(i ,  ij \bb U)} ( \mathrm y_j^{(i)}) ^* |
\leq \varphi^{D_\zeta}\Psi
\ee
\item for any $i $ and $\bb T, \bb U$: $|\bb T|+|\bb U|\leq p$,
 \be\label{559c3}
  \left| \mG^{(i\bb T,\emptyset )}_{ii}-  \frac{-1}{w(1+m^{(i\bb T,\emptyset )} )}\right|\leq \varphi^{D_\zeta}\Psi
\ee
 \end{enumerate}

Here (i) and (ii) follow from Lemma \ref{wempl};
  (iv) follows from \eqref{78} and  the case (iii) with $\bb T = \emptyset = \bb U$ follows from Lemma \ref{lem:bh} and  \eqref{110d2}. { The general case, i.e., $\bb T, \bb U\neq \emptyset$ can be proved similarly using  \eqref{37mk}.}
Furthermore, since  $|\bb T|$,$ |\bb U|\leq p$ and $p\leq \varphi^{2\zeta}$, there exists   a set $\Xi_0$ with
$$
P(\Xi_0)\geq 1-e^{-\varphi^{2 \zeta+5}}
$$
 such that for  any ${w}\in \b {\rm S}(b)$, $b>5D_\zeta$   the above properties \eqref{wwxx}-\eqref{559c3} hold  for {\it all} $|\bb T|$,$ |\bb U|\leq p$.
The reason is the number of the $\bb T$, $\bb U$ satisfying   $|\bb T|$,$ |\bb U|\leq p$ is bounded by $N^{2p}\leq \varphi^{4\zeta+1}$, where we have used \eqref{kk20n}.

 Since $\Psi$ is a monotonic in $\Lambda$,  we can
 replace $\Psi$ in \eqref{559c1}- \eqref{559c3} by  $\wt \Psi$ in the set $\Xi \cap \Xi_0$. {By  \eqref{kk20n}, we have   $\P[\Xi_0^c] \ll e^{-p_{N}(\log N)^2 }$.} For notation simplicity we will use $\Xi$ for the set $\Xi \cap \Xi_0$ from now on.  We claim that, for any $i\in A\subset\llbracket 1, N\rrbracket$,  $|A|\leq p$, there exist  decompositions

 \be \label{511new}
    \cal Q_{A}  \left(wG_{ii}\right)^{-1}
 \;=\;        { { \cal Z}}_{i, A}+   \cal  Q_{A}{\bf 1}(\Xi ^c)
   \wt { { \cal Z}}_{i, A}
 \ee
   \be
\label{511Lwu}
     Q_{A}  \left(w\mG^{(i,\emptyset)}_{ii}\right)^{-1}
     \;=\;           { { Z}}_{i, A}   +  Q_{A}{\bf 1}(\Xi ^c)\wt { { Z}}_{i, A}
      \ee
so that \eqref{511a} holds
 with $\cal Y=|w|^{-1/2}$ and $\cal X=\varphi^{D_\zeta+2
  \zeta}|w^{ 1/2}|\wt \Psi$. Notice that the  condition $\cal X<1$ follows from  $\wt \Lambda\ll  |m_c|$ and $N\eta\ge \varphi^{5D_\zeta} |m_c|$
if   $ {w}\in \b {\rm S}(b)$, $b>5D_\zeta$  is large enough. Thus  we obtain that
   \be\label{nd565}
   \E\left[ |\cal Z|^p \right]+ \E\left[ | Z^*_*|^p \right]\leq |w^{1/2}|^{ p} (Cp)^{4p}(\varphi^{2D_\zeta+4\zeta}\wt\Psi^2)^{ p}
\ee
Choosing $C_\zeta=2D_\zeta+20\zeta $, {   one can see that} \eqref{32youn} follows from \eqref{kk20n},  \eqref{nd565} and
the Markov inequality.

It remains to prove \eqref{511new} and \eqref{511Lwu}.  We prove \eqref{511new} first. For simplicity, we assume that  ${A = \{ 1, \ldots,  \abs{A}\}}$.   Denote  the first $|A|$ column of $Y_z$   by  ${\bf a}$ so that $\bf a$ is a
$N \times |A|$ matrix.
Similarly, denote by  $B $ the matrix obtained after removing  the first $K$-columns
of $Y$.   Then we  have the identity
 $$
Y^* Y - w  \;=\;
\begin{pmatrix}
\ba^*  \ba - w & \ba^*  B
\\
B^*  \ba & B^*  B - w
\end{pmatrix}\,.
$$
Recall  the identity \eqref{499}:
for any matrix $M$,
 $$
M (M^*M - w)^{-1} M^*=1+ w ( M M^*   - w)^{-1}.
$$
Then we have for $i,j\in A$
\begin{align}
G_{ij} &= \pBB{\frac{1}{\ba^* \ba - w - \ba^*  B (B^*  B - w)^{-1} B^* \ba}}_{ij}
= \pBB{ {1 \over \ba^* \ba -w - \ba^* (1+ w ( B B^*   - w)^{-1} )\,  \ba}}_{ij} \;\non \\
&= \pBB{ {1 \over - w - w\, \ba^*  \cal G^{(A, \emptyset)}  \,  \ba} }_{ij}\;,
\quad  \cal G^{(A, \emptyset)}  =   ( B B^*   - w)^{-1}.   \label{eqn:G11fin}
\end{align}

Rewrite
$$
I + \ba^* \cal G^{(A, \emptyset)}   \,  \ba
 = \al (I+ R),
\quad R :=  \al^{-1} \left( \ba^* \cal G^{(A, \emptyset)}  \,  \ba +I-\al I\right)
$$
where
$$
 \al:= \left(N^{-1} \sum_{j= 1}^N \cal G^{(A, \emptyset)}_{jj}+|z|^2  \frac{-1}{w(1+m_\mG^{(A,\emptyset)}  )}+1\right) = m_\mG^{(A,\emptyset)}  -\frac{|z|^2}{w (1+m_\mG^{(A,\emptyset)} ) }+1
 $$
{We will prove  $\|R\|\ll 1$ with high probability. }
Using \eqref{defmc1}, $\Lambda\ll m_{\rm c}$ \eqref{res:wemplnew} and \eqref{37mk},  we have
$$
\al \sim w^{-1/2}, \quad {\rm in }\; \Xi
$$
By  \eqref{559c1}, \eqref{559c3} and \eqref{37mk}, we have
$$
 \al R_{ii}=  (1-\E_{\by_i})\by_i^*  \cal G^{(A,  \emptyset)}  \by_i+|z|^2\left(\cal G^{(A, \emptyset)}_{ii}- \frac{-1}{w(1+m_\mG^{(A,\emptyset)}  )}\right)=O( \varphi^{D_\zeta}\wt \Psi), \quad {\rm in }\; \Xi,
$$
$$
\al  R_{ij}=    \by_i^*  \cal G^{(A, \emptyset)}  \by_j\leq O( \varphi^{D_\zeta}\wt \Psi), \quad {\rm in }\; \Xi.
 $$

 Therefore, we  have the bound
\be\label{Z5}
\| {\bf 1} (\Xi)  R \| =O(  \varphi^{D_\zeta} \wt \Psi \al^{-1})=O(\varphi^{D_\zeta}  |w|^{ 1/2}\wt \Psi)\ll1,
\quad
\| {\bf 1} (\Xi)   R^k \|= O(  \varphi^{D_\zeta} \wt \Psi \al^{-1})^k |A|^{k-1},\quad k=1,2,\dots
\ee
 With \eqref{eqn:G11fin} and the definition of $R$,  we have  $-w  \al  G_{ij} = [(I+R)^{-1} ]_{ij}$ for $i,j\in A$. Therefore,
$$
-w G_{ii} \al = [(I+R)^{-1} ]_{ii}
=1+  \sum_{j=1}^{|A|-1}  ((-R)^j)_{ii}
+\al w\sum_{j\in A} ((-R)^{|A|})_{ij}G_{ji}
$$
  Then, together with \eqref{Z5}, \eqref{res:wemplnew} and $m_c\sim |w^{-1/2}|\sim \al$, we
  have thus proved that, in $\Xi$,
$$
 -w G_{ii} \al = 1+  \sum_{j=1}^{{|A|} -1}   (R^j)_{ii}
+ \OO \left( {|A|}\varphi^{D_\zeta}  |w|^{1/2}\wt \Psi  \right )^{ {|A|} } , \quad {\rm in }\; \Xi
$$
Thus
\begin{align}\label{R3}
 \frac{-1}{wG_{ii} }  &=    \al  U_{A} +   \OO(|w|^{-1/2} ( |A|^2\varphi^{D_\zeta} |w|^{1/2}\wt \Psi)^{{|A|} })
 \\\nonumber&= \al  U_{A} +   \OO(|w|^{-1/2} ( |A| \varphi^{ D_\zeta+2\zeta} |w|^{1/2}\wt \Psi)^{{|A|} }) , \quad {\rm in }\; \Xi
 \end{align}
where we used $|A|\leq p\leq \varphi^{2\zeta}$ and $U_{A} $ is a linear combination of the following products of $(R^j)_{ii}$'s
$$
 \prod_k (R^{j_k})_{ii},\quad \quad\quad 0\leq  \sum_k j_k\leq {|A|} -1.
$$
Notice  we have
\be\label{88}
\cal Q_A\left( \prod_k \al (R^{j_k})_{ii}\right)= 0 ,\quad \quad\quad
\ee
provided that $0\leq  \sum_k j_k\leq {|A|} -1$.
This is because that $\al$ is independent of $\{\by_k: k\in A\} $ and  $R_{ab}$  is independent of
$\{\by_k: k\in A, k\neq a,b\}$. Hence there exists $ \ell  \in A$ such that   $\by_\ell$  does not appear in $\prod_k \al (R^{j_k})_{ii}$
and this proves \eqref{88}. Therefore, we have proved that
\be\label{nyyy}
\cal Q_A \al U_{A}=0.
\ee

Define $\Omega_A$ as the probability space for the columns $\{\by_k: k\in A\} $ and  $\Omega_{A^c}$  the one for the columns $\{\by_k: k\in A^c\} $.  Then the full probability  space $\Omega$ equals to $
\Omega= \Omega_A\times \Omega_{A^c}$.
 Define  $\pi_{A^c}$ to be  the projection  onto    $\Omega_{A^c}$ and  $\Xi^*=\left( \pi^{-1}_{A^c}\cdot \pi_{A^c}\cdot  \Xi \right)$.   Then
${\bf 1}(  \Xi ^*)$ is independent of  $\{\by_k: k\in A\} $.      Hence we can extend \eqref{nyyy} to
$$
\cal Q_A  {\bf 1}(  \Xi ^*)\al U_{A}= 0.  \quad
$$
 Let
$$
   \wt { { \cal Z}}_{i, A}=\left(wG_{ii}\right)^{-1} + {\bf 1}(\Xi^* \setminus \Xi) \al U_{A}, \quad
 \cal Z_{i, A} =  \cal Q_{A}  {\bf 1}( \Xi )  \left [ \left(  wG_{ii}\right)^{-1}  + \al U_{A} \right ]
$$
so that \eqref{511} is satisfied, i.e.,
\begin{align*}
& \cal Z_{i, A} +  \cal Q_{A} {\bf 1}( \Xi^c )  \wt { { \cal Z}}_{i, A}\\
 & = \cal Q_{A}  {\bf 1}( \Xi )  \left [ \left(  wG_{ii}\right)^{-1}  + \al U_{A} \right ]
+  \cal Q_{A} {\bf 1}( \Xi^c )  \left [  \left(wG_{ii}\right)^{-1} + {\bf 1}(\Xi^* \setminus \Xi) \al U_{A} \right ] \\
& = \left( \cal Q_{A}  wG_{ii}\right)^{-1}  +  \cal Q_{A}  \left [  {\bf 1}( \Xi )  \al U_{A}
+   {\bf 1}( \Xi^c )   {\bf 1}(\Xi^* \setminus \Xi) \al U_{A} \right ]  \\
 & = \left( \cal Q_{A}  wG_{ii}\right)^{-1}  +  \cal Q_{A}  \left [  {\bf 1}( \Xi )  \al U_{A}
+   {\bf 1}(\Xi^* \setminus \Xi) \al U_{A} \right ]  = \left( \cal Q_{A}  wG_{ii}\right)^{-1}.
\end{align*}
 By \eqref{R3}, $ |\cal Z_{i, A}| \le  \OO(|w|^{-1/2} ( |A| \varphi^{ D_\zeta+2\zeta} |w|^{1/2}\wt \Psi)^{{|A|} })$ in  $\Xi$.
We now prove that
\be\label{tqdm}
\wt {{ \cal Z}}_{i, A}= \left(wG_{ii}\right)^{-1}+ {\bf 1}(\Xi^* \setminus \Xi) \al U_{A}\leq N^{C|A|} .
\ee
By  \eqref{bdl}, we have $\left(wG_{ii}\right)^{-1}=\OO(N^C)$.
Notice that $\al$ is independent of
$\{\by_k: k\in A\} $. Since $\al\sim |w^{-1/2}|$ in $\Xi$,  the same asymptotic  holds in $\Xi^*\setminus \Xi$.
By definitions of $U_{A}$ \eqref{R3} and $R$, and the assumption $X_{ij}=O(N^{C})$, we obtain \eqref{tqdm}
and this  completes the proof of \eqref{511new}.
Similarly, we can prove \eqref{511Lwu}   and this completes the proof of Lemma \ref{lem:32y}.

  \appendix

 \section{Proof of the properties of $m_{\rm c}$ and $\rho_{\rm c}$}\label{app:macro}

  In this appendix we are going to prove the lemma \ref{pmcc1}, \ref{pmcc12} and \ref{pmcc12.5}.  We can solve $m_{\rm c}$ explicitly by the following formula.

 \begin{lemma} [Explicit expression of $m_{\rm c}$] \label{lem: eemc}
 For any $E\in\mathbb{R}$, let
$$
A_\pm:=A_\pm(E,z):=2E^{3/2}-9E^{1/2}(1+2|z|^2)\pm 6\sqrt3  |z| \sqrt{((\lambda_+-E)(E-\lambda_-))_+}.
$$
Then we have
 \be\label{mc=}
 \lim_{\eta\to 0^+}m_{\rm c}(E+\ii\eta , z)=-\frac23-\frac{1}{2^{1/3} 3\sqrt{E }}\left(
 \frac{1- \sqrt3 \ii}{ 2} A_+^{1/3}(E,z)
 + \frac{1+\sqrt3 \ii}{ 2} A_-^{1/3}(E,z)
 \right),
   \ee
where we note $x^{1/3}=\sgn(x)|x^{1/3}|$.
Moreover,  for  general $w\in \C$, $m_{\rm c}(w, z)$ is the analytic extension of $ \lim_{\eta\to 0^+}m_{\rm c}(E+\ii\eta , z) $.
 \end{lemma}

\begin{proof}[Proof of Lemma \ref{lem: eemc}] By definition, $m_{\rm c}$ is an analytic function, so we only need to prove \eqref{mc=}. By definition, $m_{\rm c}$ is one of  the three solutions of \eqref{defmc1}, and needs to have positive imaginary part. Solving explicitly this degree three polynomial equation proves that there is just one such solution, with the limit
\ref{mc=} close to the critical axis.
\end{proof}

Since $\rho_{\rm c} (E)= \frac1\pi\im m_{\rm c}(E+\ii 0^+)$, by  \eqref{mc=} and  $A_+\geq A_-$, we have:  for $0\leq E\leq \lambda_+$,
 \begin{equation}\label{rhoc=}
 \rho_{\rm c}(E, z)=\frac{1}{ 2^{4/3} 3^{1/2}\pi \sqrt{E}}\left(A_+^{1/3}-A_-^{1/3}\right)\geq 0
 \end{equation}

With  Lemma \ref{lem: eemc} and \eqref{rhoc=}, one can easily prove Proposition \ref{prorhoc}.\\

\begin{proof}[Proof of Lemma \ref{pmcc1}] By definition,
\begin{equation}\label{eqn:realPart}
\re m_{\rm c}(w, z)=\int_\R \frac{\rho_{\rm c} (x, z)(x-E) }{(x-E)^2+\eta^2} dx
\end{equation}
so for the first case this implies
$$
 0>\re m_{\rm c}(w, z)\geq
\int \frac{\rho_{\rm c} (x, z) }{x-E } \rd x.
$$

Moreover, recall that $\al=\sqrt{1+8|z|^2}$, so (still in the first case)
$$
0\geq \int \frac{\rho_{\rm c} (x, z) }{x-E } dx\geq \int \frac{\rho_{\rm c} (x, z) }{x-\lambda_+ } dx=m_{\rm c}(\lambda_+, z)=\frac{-2}{\al+3}\geq \frac{-1}{2}.
$$
We also have easily $|m_{\rm c}|\sim 1$ easily from (\ref{eqn:realPart}), we therefore obtained
the l.h.s. of \eqref{A17}.
Similarly, one can prove $\im m_{\rm c}\sim \eta$ thanks to
 $$
\im m_{\rm c}(w, z)=\eta\int_\R \frac{\rho_{\rm c} (x, z)  }{(x-E)^2+\eta^2} \rd x
$$
and complete the proof for the first case.

 For the second case,  it is easy to prove
\eqref{A18} when $w=\lambda_+$, as we did from an explicit calculation.  Then one obtains
\eqref{A18} by expanding $m_{\rm c}$ around $m_{\rm c}(\lambda_+, z)$, using \eqref{defmc1}.
The estimate \eqref{esmallfake} directly follows from \eqref{A18}.

Similarly, for the third case, first $m_{\rm c}=\infty$, i.e., $m_{\rm c}^{-1}=0$ when $w=0$, then   one can easily obtain \eqref{A20} in case 3 by solving \eqref{defmc1} with expanding $m_{\rm c}^{-1}$ around $(m_{\rm c}(0, z))^{-1}$. The estimate \eqref{A20a} directly follows from \eqref{A20}.
The fourth case follows from
\be\label{aa26}
m_{\rm c}(w, z)=\int \frac{\rho_{\rm c} (x, z)}{x-w} \rd x
\ee
and the properties of $\rho$ stated in proposition \ref{prorhoc}.
\end{proof}

\begin{proof}[Proof of Lemma \ref{pmcc12}] This is similar to the proof of Lemma  \ref{pmcc1}.
\end{proof}

\begin{proof}[Proof of Lemma \ref{pmcc12.5}]  We are going to  prove this lemma in   the case $|z|\leq 1-\tau$, the other cases  can be proved similarly.  Note first that \eqref{dA2} is a consequence of all possible cases in Lemma \ref{pmcc1}.

We now prove \eqref{26ssa} in the four different cases, which have been classified in Lemma \ref{pmcc1}.  In   the first case, if
additionally $\eta\sim 1 $, as $0>\re(m_{\rm c})>-1/2$,  the l.h.s. in \eqref{26ssa} is bounded by $\OO(1)$, which implies \eqref{26ssa}.  For  the first case if $\eta  $ is small enough, since $|\re w|\sim (1+m_{\rm c})\sim  1$  and $|\im (m_{\rm c})| \sim\eta$, so
\be\label{zgzgt}
  \im \frac{1}{ w(1+m_{\rm c})}\leq C\ |\im (w(m_{\rm c}+1))|\leq C\im m_{\rm c}
\ee
which gives \eqref{26ssa} in the first case. In the same way we get \eqref{26ssa} in the second case, where   $\im m_{\rm c}\geq c\eta$. For the third case, using \eqref{A20}, one can easily prove \eqref{26ssa}.
Finally,  the fourth case  is simple since the l.h.s. in \eqref{26ssa} is clearly $\OO(1)$.

We now prove \eqref{ny27}.  Using \eqref{A20a} and \eqref{A21}, ($\al=\sqrt{1+8|z|^2} $ is a real number) we  have that,  in the cases  three and four,
 \be\label{A31}
  \left|(-1 + |z^2|)
   \left( m_{\rm c}-\frac{-2}{3+\al}\right)  \left( m_{\rm c}-\frac{-2}{3-\al}\right)\right| \geq C |\im m_{\rm c}|^2\ge
C\, |w|^{-1}
 \ee
 For case two, using \eqref{A18},
  \be\label{slzts}
  \left|(-1 + |z^2|)
   \left( m_{\rm c}-\frac{-2}{3+\al}\right)  \left( m_{\rm c}-\frac{-2}{3-\al}\right)\right|\geq
    C \left|  m_{\rm c}-\frac{-2}{3+\al}\right | \geq C\/\left|\frac{\sqrt{\kappa+\eta}}{w}\right|
 \ee
 Note $m_{\rm c}(\lambda_+)= - 2/(3+\al) $.  For case one, with \eqref{aa26}, it is easy to prove that either $\im m_{\rm c}\sim 1$ or $\re  m_{\rm c}-m_{\rm c}(\lambda_+)=\re m_{\rm c}+ 2/(3+\al) \sim 1$. It implies that $\left|  m_{\rm c}-\frac{-2}{3+\al}\right |\sim 1$. This completes the proof.
 \end{proof}

\section{Perturbation theorem}
In this section, we introduce the theorem on the relations between the  Green function  $G$   of  the matrix  $H$ and the  Green function of the minor of the   matrix. This theorem was proved in \cite{ErdYauYin2010PTRF}.  We first introduce some notations (here we use $[]$ instead of $()$ in \cite{ErdYauYin2010PTRF}, since upper index $()$ has been used in the main part of the paper).
  \begin{definition}\label{basicd}

Let $H$  be $N\times N$ matrix,  ${\T} \subset \llbracket 1, N\rrbracket$  and
 $H^{[\bT]}$ be the $N-|\T|$ by $N-|\T|$ minor of $H$ after removing the
 $i$-th   rows and columns index by $i\in \T$.   For $\bT=\emptyset$, we define $H^{(\emptyset)}=H$.
   For any ${\T}\subset \llbracket 1, N\rrbracket$ we introduce the following notations:
 \begin{align}
 G^{[{\T}]}_{ij}:=&(H^{[{\T}]}-w)^{-1}(i,j) ,\qquad i,j\not\in\T\non\\
 Z^{[{\T}]}_{ij}:=& =\sum_{k,\ell\notin {\T}}
 h_{ik}
 G^{[{\T}]}_{k \ell}h_{\ell j\,} \non \\
\wH^{[{\T}]}_{ij}:= & h_{ij}-w\delta_{ij}-Z^{[{\T}]}_{ij}.
 \end{align}
\end{definition}

The following  formulas  were proved in
  Lemma 4.2  from \cite{ErdYauYin2010PTRF}.

\begin{lemma}[Self-consistent perturbation formulas] \label{basicIG} Let
$\bT\subset \llbracket 1, N\rrbracket$. For  simplicity, we use the
 notation $[i \,{\T}]$ for $[\{i\}\cup {\T}]$ and $[i j \,{\T}]$
 for $[\{i,j\}\cup {\T}]$.
 Then we have the following identities:
\par\begin{enumerate}
\item For any  $i\notin {\T}$
\be\label{1}
 G^{[{\T}]}_{ii}=(\wH^{[i\,{\T}]}_{ii})^{-1}.
 \ee
 \item For $i\neq j$ and $i,j\notin {\T}$
\be\label{2}
 G^{[{\T}]}_{ij}=-G^{[{\T}]}_{jj}G_{ii}^{[j\,{\T}]}\wH^{[ij\,\,{\T}]}_{ij}=
-G^{[{\T}]}_{ii}G_{jj}^{[i\,{\T}]}\wH^{[ij\,\,{\T}]}_{ij}.
 \ee
   \item  For any indices
  $i,j,k \notin {\T}$ with $k \not \in \{i , j\}$  (but $i = j$ is allowed)
 \be\label{3}
G^{[{\T}]}_{ij}-G^{[k\,\,{\T}]}_{ij}=G^{[{\T}]}_{ik}G^{[{\T}]}_{kj}
(G^{[{\T}]}_{kk})^{-1} . \ee
 \end{enumerate}
 \end{lemma}

\section{Large deviation estimates. }
In order to obtain the self-consistent equations for the Green functions, we needed  the following large deviation
estimate.

\begin{lemma}[Large deviation estimate]\label{lem:bh} For any $\zeta>0$, there exists $Q_\zeta>0$ such that  for $\bb T\subset\llbracket 1, N\rrbracket$, $|\bb T| \leq N/2$ the following estimates hold  \hp{\zeta}:
\begin{align}\label{130}
| Z^{(\bb T)}_{i }|=
\left|(1-\E_{ \mathrm y_i})  \left(\mathrm y_i^{(\bb T)} G^{(\bb T,i )}    \mathrm y_i^{(\bb T)*}\right)  \right|
 \leq  \varphi^{  Q_\zeta/2} \sqrt{\frac{\im m_ G^{(\bb T,i )}+ |z|^2 \im G^{(\bb T,i )}_{ii} }{N\eta}}, \\
|\cal Z^{(\bb T)}_{i }| =
 \left|(1-\E_{ \by_i})   \left(\by_i^{(\bb T)*}  \mG^{(i , \bb T )}  \by_i^{(\bb T)} \right) \right|
\leq  \varphi^{  Q_\zeta/2} \sqrt{\frac{\im m_\mG^{(i,\bb T)}+ |z|^2 \im\mG^{(i, \bb T)}_{ii} }{N\eta}}. \non
 \end{align}
 Furthermore, for $i\neq j$, we have
\begin{align}\label{132}
\left|
(1-\E_{\mathrm y_i\mathrm y_j})
\left( \mathrm y_i^{(\bb T)} G^{(\bb T,ij)}   \mathrm y_j^{(\bb T)*}\right)
 \right|
& \leq
  \varphi^{  Q_\zeta/2}
  \sqrt{\frac{\im m_ G^{(\bb T,ij)}
  +|z|^2\im G^{(\bb T,ij)}_{ii}+|z|^2 \im G^{(\bb T,ij)}_{jj}}{N\eta}},
  \\
  \left|
(1-\E_{\by_i\by_j} )
\left(\by_i^{(\bb T)*}   \mG^{(ij,\bb T)}  \by_j^{(\bb T)}\right)
 \right|
& \leq
  \varphi^{  Q_\zeta/2}
  \sqrt{\frac{\im m_\mG^{(ij,\bb T)}
  +|z|^2\im\mG^{(ij,\bb T)}_{ii}+|z|^2 \im\mG^{(ij,\bb T)}_{jj}}{N\eta}}, \label{1321}
 \end{align}
 where
 \be\label{1328}
 \E_{\mathrm y_i\mathrm y_j} \left( \mathrm y_i^{(\bb T)} G^{(\bb T,ij)}   \mathrm y_j^{(\bb T)*}\right)
= |z|^2G^{(\bb T,ij)}_{ij}, \quad
\E_{\by_i\by_j}\left(\by_i^{(\bb T)*}   \mG^{(ij,\bb T)}  \by_j^{(\bb T)}\right)= |z|^2\mG^{(ij,\bb T)}_{ij}.
 \ee
 \end{lemma}

We first recall the following  large deviation
estimates concerning independent random variables, which were proved in
Appendix B of \cite{ErdYauYin2010PTRF}.

\begin{lemma}\label{LD}
Let $a_i$ ($1\leq i\leq N$) be independent complex random  variables with mean zero,
variance $\sigma^2$  and having a uniform  subexponential decay
$$
\P(|a_{i}|\geq x \sigma)\leq \ttau^{-1} \exp{\big( - x^\ttau\big)}, \qquad \forall \; x\ge 1,
$$
with some $\ttau>0$.
Let $A_i$, $B_{ij}\in \C$ ($1\leq i,j\leq N$).
Then   there exists a constant $0< \phi<1$, depending on $\ttau $,
 such that for any $\xi > 1$
we have
 \begin{align}
\P\left\{\left|\sum_{i=1}^N a_iA_i\right|\geq  (\log N)^{\xi
}
 \sigma \,\Big(\sum_{i}|A_i|^2\Big)^{1/2}\right\}\leq &\; \exp{\big[-(\log N)^{\phi \xi
  } \big]},
\label{resgenHWTD} \\
\P\left\{\left|\sum_{i=1}^N\overline a_iB_{ii}a_i-\sum_{i=1}^N\sigma^2 B_{ii}\right|\geq
(\log N)^{\xi
} \sigma^2 \Big( \sum_{i=1}^N|B_{ii}|^2\Big)^{1/2}\right\}\leq &\;
  \exp{\big[-(\log N)^{\phi \xi
 } \big]},\label{diaglde}\\
\P\left\{\left|\sum_{i\neq j}\overline a_iB_{ij}a_j\right|\geq (\log N)^{\xi
} \sigma^2
\Big(\sum_{i\ne j} |B_{ij}|^2 \Big)^{1/2}\right\}\leq & \; \exp{\big[-(\log N)^{\phi \xi
 } \big]}
 \label{resgenHWTO}
\end{align}
for any sufficiently large $N\ge N_0$, where $N_0=N_0(\ttau)$ depends on  $\ttau$.
\end{lemma}

\begin{proof}[Proof of Lemma \ref{lem:bh}] We will  only prove the assertion of this lemma concerning
the Green function $G$. Similar statement for $\mG$ can be proved with the row-column symmetry. From now on, we will only prove
all statements concerning $G$ if identical proofs are valid for $\mG$ and we will not repeat this comment.

We first prove   \eqref{130} by writing
\begin{align}\label{j1}
& (1-\E_{\mathrm y_i})\left( \mathrm y^{(\bb T)}_i G^{(\bb T,i)} \mathrm  y^{(\bb T)*}_i\right)
\\\nonumber
= &(1-\E_{\mathrm y_i}) |z|^2  G^{(\bb T, i)}_{ii}
 - (1-\E_{\mathrm y_i})    \sum_k \Big  [  z   G^{(\bb T,i)}_{ik}   X^*_{ik}
 +   z^* X_{ik}  G^{(\bb T, i)}_{ki} \Big ]  +
  (1-\E_{ \mathrm y_i})
   \sum_{j k}    X_{ ij}  G^{(\bb T,i)}_{jk}  X^*_{ki}
\end{align}
with $Y=X-zI$.
Since  $G^{(\bb T, i)}_{ii}$  is independent of $\mathrm y_i$, the first term on the right hand side vanishes.
  For any $\zeta>0$, we apply \eqref{diaglde} and \eqref{resgenHWTO} in Lemma \ref{LD} with $\phi \xi = \zeta \log \log N$.  Denote  $\xi = Q_\zeta/2$
and  the last term in \eqref{j1} is bounded by
$$
\varphi^{Q_\zeta /2 }\sqrt { N^{-2}  \sum_{jk} | G^{(\bb T,i)}_{jk}|^2 }
\le \varphi^{Q_\zeta /2 } \sqrt{\frac{ \im  m_ G^{(\bb T,i)} }{N\eta}}
$$
\hp{\zeta}. \nc  Similarly, with \eqref{resgenHWTD}, the second term on the right hand side is bounded by
$$
\varphi^{  Q_\zeta/2} |z| \sqrt {  N^{-1}  \sum_{k }\left(| G^{(\bb T,i )}_{ik}|^2+
| G^{(\bb T,i )}_{ki}|^2\right)} \le \varphi^{  Q_\zeta/2}   \sqrt{ \frac{|z^2|\im  G^{(\bb T,i )}_{ii}  }{N\eta}}
$$
The proofs for the other bounds follow from similar arguments.
\end{proof}

\section{Abstract decoupling lemma}

{ We recall  an abstract cancellation Lemma proved in \cite{PilYin2011}.}

\begin{lemma} \label{abstractZlemma}  Let $\cal I$ be a finite set which may depend on $N$ and
$$
\cal I_i \subset \cal I,\quad 1\leq i\leq N.
$$
Let ${S}_1, \dots, {S}_N$ be random variables which depend on the independent random variables $\{x_\al, \al \in \cal I \}$.
 In application, we often take  $\cal I = \llbracket 1, N\rrbracket$ and $\cal I_i = \{ i \}$.

Recall  $\E_i$ denote  the conditional expectation with  respect to the complement of  $\{x_\al, \al \in \cal I_i \}$, i.e., we integrate out the variables
$\{x_\al, \al \in \cal I_i \}$.
Define the commuting projection operators
$$
Q_i =  1 - P_i, \; P_i = \E_i, \quad P_i^2 = P_i,   \; Q_i^2 = Q_i, \quad [Q_i, P_j]=[P_i, P_j]=[Q_i, Q_j]=0 \, .
$$
For $A\subset\llbracket 1, N\rrbracket$
$$
Q_A:=\prod_{i\in A}Q_i,\quad P_A:=\prod_{i\in A}P_i
$$
We  use  the notation
\begin{align*}
[{\bf Z}] \;=\; \frac{1}{N} \sum_{i = 1}^N {\bf Z}_i,\quad {\bf Z}_i:= Q_i{S}_i\,.
\end{align*}
Let  $p $ be an even integer
 Suppose  for some constants $C_0$, $c_0>0$ there is a set $\Xi$ (the  "good configurations") so that the
following assumptions hold:

\begin{enumerate}

\item (Bound on $Q_A  S_i$ in $\Xi$). There exist   deterministic positive numbers $\cal X<1$ and $\cal Y$ such that  for any set $A\subset\llbracket 1, N\rrbracket$ with  $i\in  A$ and $\abs{A } \leq p$,
  $Q_{A}S_i$ in $\Xi$  can be written as the sum of two  random variables
  {
 \be\label{511}
   ( Q_{A} S_i  )=  {\bf Z}_{i, A}+  Q_{A}{\bf 1}(\Xi ^c)
   \wt { {\bf Z}}_{i, A}, \quad  {\rm in}\quad   \Xi
  \ee
  and
   \be\label{511a}
   \; | {\bf Z}_{i, A} |\leq  \cal Y \big(C_0\cal X|A| \big)^{ |A|} ,\quad
   | \wt {\bf Z}_{i, A} |\leq   \cal Y  N^{C_0|A|} \ee
}

\item (Crude  bound on $  S_i$).
$$
\max_i | S_i | \;\leq\; \cal Y N^{C_0}\,.
$$

\item ($\Xi$ has high probability).
$$
  \P[\Xi^c] \;\leq\; \me^{-c_0(\log N)^{3/2}  p }\,.
$$

   \end{enumerate}

Then, under the assumptions (i) -- (iii),      we have
$$
\E [{\bf Z}] ^{p}  \leq
  (Cp)^{4p }
 \big[ \cal X^{2} + N^{-1}\big]^{p} \cal Y^p
$$
for some $C>0$ and any sufficiently large $N
$.
 \end{lemma}

Roughly speaking, this lemma increase the estimate of $ {\bf Z}_i$ from $\cal X$
to $\cal X^2$ after averaging over $i$.


\begin{bibdiv}

 \begin{biblist}




\bib{Bai1997}{article}{
   author={Bai, Z. D.},
   title={Circular law},
   journal={Ann. Probab.},
   volume={25},
   date={1997},
   number={1},
   pages={494--529}
}


\bib{BaiSil2006}{book}{
   author={Bai, Z. D.},
   author={Silverstein, J.},
   title={Spectral Analysis of Large Dimensional Random Matrices},
   series={Mathematics Monograph Series},
   volume={2},
   publisher={Science Press},
   place={Beijing},
   date={2006}}





\bib{BenCha2011}{article}{
   author={Benaych-Georges, F.},
   author={Chapon, F.},
   title={Random right eigenvalues of Gaussian quaternionic matrices},
   journal={Random Matrices: Theory and Applications},
   volume={2}
   date={2012}}



\bib{BorSin2009}{article}{
   author={Borodin, A.},
   author={Sinclair, C. D.},
   title={The Ginibre ensemble of real random matrices and its scaling
   limits},
   journal={Comm. Math. Phys.},
   volume={291},
   date={2009},
   number={1},
   pages={177--224}
}




\bib{CacMalSch2012}{article}{
   author={Cacciapuoti, C.},
   author={Maltsev, A.},
   author={Schlein, B.},
   title={Local Marchenko-Pastur law at the hard edge of sample covariance matrices},
   journal={to appear in Journal of Mathematical Physics},
   date={2012}}

\bib{Dav1995}{article}{
   author={Davies, E. B.},
   title={The functional calculus},
   journal={J. London Math. Soc. (2)},
   volume={52},
   date={1995},
   number={1},
   pages={166--176}}





\bib{Ede1997}{article}{
   author={Edelman, A.},
   title={The probability that a random real Gaussian matrix has $k$ real
   eigenvalues, related distributions, and the circular law},
   journal={J. Multivariate Anal.},
   volume={60},
   date={1997},
   number={2},
   pages={203--232}
}





\bib{ErdYauYin2010PTRF}{article}{
   author={Erd{\H{o}}s, L.},
   author={Yau, H.-T.},
   author={Yin, J.},
   title={Bulk universality for generalized Wigner matrices},
   journal={Probability Theory and Related Fields},
  volume={154}
  number={1-2}
  pages={341--407} 
  date={2012}
}


\bib{ErdYauYin2010Adv}{article}{
   author={Erd{\H{o}}s, L.},
   author={Yau, H.-T.},
   author={Yin, J.},
   title={Rigidity of Eigenvalues of Generalized Wigner Matrices},
   journal={Adv. Mat.},
   date={2012},
   volume={229},
   number={3},
   pages={1435--1515}
}



\bib{For2010}{book}{
   author={Forrester, P. J.},
   title={Log-gases and random matrices},
   series={London Mathematical Society Monographs Series},
   volume={34},
   publisher={Princeton University Press},
   place={Princeton, NJ},
   date={2010},
   pages={xiv+791}}

\bib{ForNag2007}{article}{
     author={Forrester, P. J.},
     author={Nagao, T.},
     title={Eigenvalue Statistics of the Real Ginibre Ensemble},
     journal={Phys. Rev. Lett.},
     volume={99},
     date={2007}
}



\bib{Gin1965}{article}{
   author={Ginibre, J.},
   title={Statistical ensembles of complex, quaternion, and real matrices},
   journal={J. Mathematical Phys.},
   volume={6},
   date={1965},
   pages={440--449}}

\bib{Gir1984}{article}{
   author={Girko, V. L.},
   title={The circular law},
   language={Russian},
   journal={Teor. Veroyatnost. i Primenen.},
   volume={29},
   date={1984},
   number={4},
   pages={669--679}
}

\bib{GotTik2010}{article}{
   author={G{\"o}tze, F.},
   author={Tikhomirov, A.},
   title={The circular law for random matrices},
   journal={Ann. Probab.},
   volume={38},
   date={2010},
   number={4},
   pages={1444--1491}}






\bib{GuiKriZei2009}{article}{
   author={Guionnet, A.},
   author={Krishnapur, M.},
   author={Zeitouni, O.},
   title={The single ring theorem},
   journal={Ann. of Math.}
   volume={174},
   date={2011},
   number={2},
   pages={1189--1217}}






\bib{Meh2004}{book}{
   author={Mehta, M.},
   title={Random matrices},
   series={Pure and Applied Mathematics (Amsterdam)},
   volume={142},
   edition={3},
   publisher={Elsevier/Academic Press, Amsterdam},
   date={2004},
   pages={xviii+688}}




\bib{PanZho2010}{article}{
   author={Pan, G.},
   author={Zhou, W.},
   title={Circular law, extreme singular values and potential theory},
   journal={J. Multivariate Anal.},
   volume={101},
   date={2010},
   number={3},
   pages={645--656}
}

\bib{PilYin2011}{article}{
    author={Pillai, N.},
    author={Yin, J.},
    title={Universality of Covariance matrices},
    journal={preprint {\tt arXiv:1110.2501}},
    date={2011}
    }



\bib{Rud2008}{article}{
author={Rudelson, M.},
title={Invertibility of random matrices: Norm of the inverse},
journal={Ann. of Math.},
volume={168},
number={2} ,
date={2008},
pages={575--600}}


\bib{RudVer2008}{article}{
   author={Rudelson, M.},
   author={Vershynin, R.},
   title={The Littlewood-Offord problem and invertibility of random
   matrices},
   journal={Adv. Math.},
   volume={218},
   date={2008},
   number={2},
   pages={600--633}
}



\bib{Sin2007}{article}{
   author={Sinclair, C. D.},
   title={Averages over Ginibre's ensemble of random real matrices},
   journal={Int. Math. Res. Not. IMRN},
   date={2007},
   number={5}
}


\bib{TaoVu2008}{article}{
   author={Tao, T.},
   author={Vu, V.}
   title={Random matrices: the circular law},
   journal={Commun. Contemp. Math.},
   volume={10},
   date={2008},
   number={2}
   pages={261--307},
}



\bib{TaoVuKri2010}{article}{
   author={Tao, T.},
   author={Vu, V.},
   title={Random matrices: universality of ESDs and the circular law},
   note={With an appendix by Manjunath Krishnapur},
   journal={Ann. Probab.},
   volume={38},
   date={2010},
   number={5},
   pages={2023--2065}}










 \end{biblist}

\end{bibdiv}
\end{document}